\documentclass[final, nomarks]{dmtcs-episciences}

\usepackage[utf8]{inputenc}
\usepackage{subcaption}
\usepackage{tikz}
\usepackage{amsmath}
\usepackage{doi}

\usepackage[round]{natbib}

\newtheorem{theorem}{Theorem}
\newtheorem{lemma}[theorem]{Lemma}
\newtheorem{definition}[theorem]{Definition}

\newtheorem{corollary}[theorem]{Corollary}
\newtheorem{example}[theorem]{Example}
\newtheorem{remark}[theorem]{Remark}
\newtheorem{proposition}[theorem]{Proposition}

\author{Meike Weiß \and Alice C. Niemeyer}
\title[Strong Embeddings of Cubic Planar Graphs]{Strong Embeddings of 3-Connected Cubic Planar Graphs on Surfaces of non-negative Euler Characteristic}
\affiliation{
  RWTH Aachen University, Aachen, Germany}
\keywords{strong embeddings, facial walks, planar, cubic}

\begin{document}

\publicationdata{vol. 28:3}{2026}{10}{10.46298/dmtcs.16872}{2025-11-06; 2025-11-06; 2026-05-26}{2026-08-11}

\maketitle
\begin{abstract}
    Whitney proved that $3$-connected planar graphs admit a unique embedding on the sphere. In contrast, Enami investigated embeddings of $3$-connected cubic planar graphs on non-spherical surfaces with non-negative Euler characteristic. He established that such an embedding exists if and only if the dual graph contains a particular subgraph. 
    Here, strong embeddings are investigated motivated by the cycle double cover conjecture and the relation to triangulated surfaces. We provide a complete characterization of strong embeddings on the projective plane, the torus, and the Klein bottle in terms of a distinguished subset of Enami’s subgraphs.
    This characterization not only deepens the structural understanding of graph embeddings on non-spherical surfaces, but also establishes a robust foundation for computing cycle double covers. As a direct consequence, we derive explicit criteria that determine when a graph does not admit a strong embedding on these surfaces—offering new tools for both theoretical analysis and algorithmic applications.
\end{abstract}
\section{Introduction} \label{sec: intro}
Combinatorial structures such as graphs are ubiquitous in mathematics. An important part of graph theory is topological graph theory, which considers embeddings of graphs on surfaces.
One of the earliest foundational results in this area was established by \cite{Whitney}, who proved that every $3$-connected planar graph admits a unique embedding on the sphere.
Building on this foundational result, subsequent research has extended embeddings of planar graphs on surfaces beyond the sphere, including those of higher genus and non-orientable types, as explored by \cite{MOHAR199687} and \cite{MR1375105}.
\cite{EnamiEmbeddings} considers embeddings of $3$-connected cubic planar graphs on the projective plane, the torus and the Klein bottle. His approach involves re-embedding planar graphs on these non-spherical surfaces and introduces an enumeration algorithm capable of computing the number of all such re-embeddings. Notably, the algorithm guarantees polynomial-time complexity with respect to the number of vertices of the input graph.
A particularly important class of graph embeddings—closely related to the \textbf{cycle double cover conjecture} \cite[]{CycleDoubleCoverConjecture}—are known as \textbf{strong graph embeddings}, or circular embeddings.
These are embeddings in which every facial walk forms a cycle. \cite{RICHTER1994273} showed that every $3$-connected planar graph admits a strong embedding on a non-spherical surface. Interestingly, certain cubic graphs require surfaces of arbitrarily high genus to admit strong embeddings, as demonstrated by \cite{MoharStrong}.

Our results on strong graph embeddings are motivated by applications to \textbf{simplicial surfaces} and establish a foundation for further significant developments in the theory of strong embeddings.
Roughly speaking, simplicial surfaces describe the incidence structure of special triangulated surfaces.
The undirected graph formed by the edges and faces of a simplicial surface is called the \textbf{face graph}, which is necessarily cubic due to the assumption of triangulated surfaces.
Given a cubic graph $G$, our goal is to compute all simplicial surfaces which have $G$ as their face graph. This task is equivalent to determining all strong graph embeddings of $G$. The results presented here represent a significant step towards achieving our goal.

Based on the results of \cite{EnamiEmbeddings} we prove a characterization for strong embeddings of $3$-connected cubic planar graphs on the projective plane, the torus and the Klein bottle as stated in Theorem \ref{theorem:main}. This strengthening necessitated a more rigorous treatment of the embedding scheme and the resulting facial walks. For this purpose, let $K_n$ denote the complete graph on $n$ vertices and, for a positive integer $k\geq 2$, let $K_{n_1,n_2,\dots, n_k}$ denote the complete $k$-partite graph with $k$ partition sets $V_1, V_2,\dots, V_k$ such that $\vert V_i\vert=n_i$ for $1\leq i \leq k$. 
We further denote by $A_3, A_5$ and $A_6$ the graphs shown in Figure~\ref{fig:SubgraphsKleinBottle}. Throughout the paper, $G^\ast$ always refers to the \textbf{dual graph} with respect to the planar embedding of a $3$-connected planar graph $G$. Whitney’s theorem ensures that this dual graph is unique and, moreover, simple, planar, and 3‑connected. With this notation in place, we state the main result of this paper.

\begin{theorem} \label{theorem:main}
Let $G$ be a $3$-connected cubic planar graph. There exists a one-to-one correspondence between inequivalent strong embeddings of $G$ on
    \begin{itemize}
        \item[1)] the projective plane and subgraphs of $G^{\ast}$ that are isomorphic to $K_4$.
        \item[2)] the torus and subgraphs of $G^{\ast}$ that are isomorphic to $K_{2,2,2}$ or isomorphic to $K_{2,2m}$ for $m\geq 1$, where for $K_{2,2m}$ the vertices in the partition sets of size two are not adjacent in $G^{\ast}$.
        \item[3)] the Klein bottle and subgraphs of $G^{\ast}$ that are isomorphic to $A_3, A_5, A_6$ or $K_{2,2m-1}$ for $m\geq 2$, where for $K_{2,2m-1}$ the vertices in the partition set of size two are not adjacent in $G^{\ast}$.
    \end{itemize}
\end{theorem}
\begin{figure}[ht!]
    \centering
    \begin{subfigure}{0.32\textwidth}
        \centering
        \begin{tikzpicture}[scale=1.8]
            \tikzset{knoten/.style={circle,fill=black,inner sep=0.7mm}}
            \node [knoten] (a) at (0,0) {};
            \node [knoten] (b) at (1,0) {};
            \node [knoten] (c) at (0.5,0.5) {};
            \node [knoten] (d) at (0.5,1) {};
            \node [knoten] (e) at (1-0.25,1) {};
            \node [knoten] (f) at (2-0.25,1) {};
            \node [knoten] (g) at (1.5-0.25,0.5) {};
            \node [knoten] (h) at (1.5-0.25,0) {};
            
            \draw[-,thick] (a) to (b);
            \draw[-,thick] (a) to (c);
            \draw[-,thick] (a) to (d);
            \draw[-,thick] (b) to (c);
            \draw[-,thick] (b) to (d);
            \draw[-,thick] (c) to (d);
            
            \draw[-,thick] (e) to (f);
            \draw[-,thick] (e) to (g);
            \draw[-,thick] (e) to (h);
            \draw[-,thick] (f) to (g);
            \draw[-,thick] (f) to (h);
            \draw[-,thick] (g) to (h);
        
        \end{tikzpicture}
        \caption{$A_3$}
    \end{subfigure}
    \begin{subfigure}{0.32\textwidth}
        \centering
        \begin{tikzpicture}[scale=1.8]
            \tikzset{knoten/.style={circle,fill=black,inner sep=0.7mm}}
            \node [knoten] (a) at (0,0) {};
            \node [knoten] (b) at (0,1) {};
            \node [knoten] (c) at (0.5,0.5) {};
            \node [knoten] (d) at (1,0.5) {};
            \node [knoten] (e) at (1.5,0.5) {};
            \node [knoten] (f) at (2,1) {};
            \node [knoten] (g) at (2,0) {};
            
            \draw[-,thick] (a) to (b);
            \draw[-,thick] (a) to (c);
            \draw[-,thick] (a) to (d);
            \draw[-,thick] (b) to (c);
            \draw[-,thick] (b) to (d);
            \draw[-,thick] (c) to (d);
            
            \draw[-,thick] (e) to (f);
            \draw[-,thick] (e) to (g);
            \draw[-,thick] (e) to (d);
            \draw[-,thick] (f) to (g);
            \draw[-,thick] (f) to (d);
            \draw[-,thick] (g) to (d);
        \end{tikzpicture}
        \caption{$A_5$}
    \end{subfigure}
    \begin{subfigure}{0.32\textwidth}
        \centering
        \begin{tikzpicture}[scale=1.8]
            \tikzset{knoten/.style={circle,fill=black,inner sep=0.7mm}}
            \node [knoten] (a) at (0,0) {};
            \node [knoten] (b) at (0.5,0) {};
            \node [knoten] (c) at (1,0.5) {};
            \node [knoten] (d) at (1,-0.5) {};
            \node [knoten] (e) at (1.5,0) {};
            \node [knoten] (f) at (2,0) {};
            
            \draw[-,thick] (a) to (b);
            \draw[-,thick] (a) to (c);
            \draw[-,thick] (a) to (d);
            \draw[-,thick] (b) to (c);
            \draw[-,thick] (b) to (d);
            \draw[-,thick] (c) to (e);
            \draw[-,thick] (c) to (f);
            \draw[-,thick] (d) to (e);
            \draw[-,thick] (d) to (f);
            \draw[-,thick] (e) to (f);
        \end{tikzpicture}
        \caption{$A_6$}
    \end{subfigure}
    \caption{The three graphs $A_3$ (a), $A_5$ (b) and $A_6$ (c)}
    \label{fig:SubgraphsKleinBottle}
\end{figure}

This main result introduces a novel method for computing strong embeddings and cycle double covers. Building on this framework, we implemented an algorithm for computing strong embeddings on surfaces of non-negative Euler characteristic within the computer algebra system \cite{GAP4}. These methods are available in the SimplicialSurfaces GAP package by \cite{simplicialsurfacegap}.
Furthermore, the underlying characterization forms the basis for existence results concerning strong embeddings of specific $3$-connected cubic planar graphs. The theoretical framework of visited edges established here extends naturally to embedding problems on surfaces of higher genus. In a forthcoming paper, we build on these findings to demonstrate that the dual graphs of Apollonian networks precisely characterize the class of $3$-connected cubic planar graphs that admit a unique strong embedding on an orientable surface \cite[]{weißakpanya}.

\section{Preliminaries} \label{sec: prelim}
In this section, we first consider graph embeddings and review a combinatorial way of describing them. Then we examine embeddings of $3$-connected cubic planar graphs on non-spherical surfaces of non-negative Euler characteristic.
We assume that all our graphs are undirected, connected, simple and finite. Let $G$ be such a graph. Then we denote by $V(G)$ the vertices of $G$ and by $E(G)$ the edges of $G$ which are 2-subsets of $V(G)$.

\subsection{Graph Embeddings}\label{sec:graphembeddings}
We adopt the standard terminology of topological graph theory as established by \cite{TopologicalGraphTheory,GraphsOnSurfaces}.
A \textbf{surface} is understood to be a compact two-dimensional manifold without boundary. An \textbf{embedding} of a graph $G$ on a surface $S$ refers to a drawing of $G$ on $S$ without edge crossings. Formally, we regard an embedding as an injective continuous map $\beta:G\rightarrow S$. The \textbf{cells} of $\beta(G)$ are defined as the connected components of $S\setminus\beta(G)$. Throughout this paper, we assume that all embeddings are cellular embeddings, meaning that each cell is homeomorphic to an open disc. For convenience, we often identify $G$ with its image under $\beta$, assuming it is already embedded on $S$.
A closed walk of an embedded graph is called a \textbf{facial walk} if it traverses the boundary of a cell. If a facial walk has distinct vertices and edges, it is a \textbf{facial cycle}. Note that each edge of $G$ is contained exactly twice in the set of facial walks. If all facial walks are facial cycles, then each edge is contained in exactly two facial cycles.
We consider cycles to be identical up to cyclic permutation and reversal.
If $G$ is a planar graph, we call the cells of the embedding on the plane \textbf{faces}.
To describe a graph embedding combinatorially, we employ rotation systems and signatures. Since they play an important role in this work, we repeat the definitions here.

\begin{definition}
    Suppose that a graph $G$ is embedded on an orientable surface. A \textbf{\emph{rotation}} $\rho_v$ around a vertex $v\in V(G)$ is a circular ordering of the edges incident to $v$ such that $\rho_v(e)$ is the clockwise successor of $e$ around $v$. Analogously, $\rho^{-1}_v(e)$ is the anticlockwise successor of $e$ around $v$. A \textbf{\emph{rotation system}} of the embedded graph $G$ is defined by $\rho=\{\rho_v\mid v\in V(G)\}$.
\end{definition}

Two rotation systems $\rho$ and $\rho'$ of $G$ are equivalent if either they are the same or if $\rho'_v=\rho^{-1}_v$ for all $v\in V(G)$.
A rotation system of a graph $G$ uniquely describes the embedding of $G$ on an orientable surface, and an embedding of $G$ on an orientable surface uniquely describes a rotation system up to equivalence. The situation, however, is different for non-orientable surfaces, where more information is needed. Therefore, the notion of rotation systems is extended to embedding schemes.

\begin{definition}\label{def:embeddingscheme}
    Let $\beta(G)$ be an embedding of $G$ on a not necessarily orientable surface and $\rho=\{\rho_v\mid v\in V(G)\}$ a rotation system of $G$.
    A \textbf{\emph{signature}} on $E(G)$ is a map $\lambda: E(G)\rightarrow \{-1,1\}$. For an edge $e=\{v,w\}$ we define $\lambda(e)=1$ if $\rho_v(e), e$ and $\rho_w^{-1}(e)$ are included in a facial walk of $\beta(G)$ as a subwalk; otherwise $\rho_v(e),e$ and $\rho_w(e)$ are included in a facial walk of $\beta(G)$ as a subwalk and we define $\lambda(e)=-1$. The pair $(\rho,\lambda)$ is called the \textbf{\emph{embedding scheme}} of $\beta(G)$ and an edge $e\in E(G)$ is called \textbf{\emph{twisted}} if $\lambda(e)=-1$ and \textbf{\emph{untwisted}} if $\lambda(e)=1$.
\end{definition}

Unlike in the orientable case, an embedding scheme is not uniquely determined by an embedding of $G$. For instance, if we change the clockwise ordering at a vertex $v\in V(G)$ to anticlockwise and for all edges $e$ incident to $v$ we change $\lambda(e)$ to $-\lambda(e)$, the corresponding embedding does not change. This change in the embedding scheme is referred to as a \textbf{local change}.
Two embedding schemes are called \textbf{equivalent} if they can be obtained from one another by a sequence of local changes. Note that an embedding scheme defines an embedding of a given graph $G$ on a non-orientable surface if and only if there is a cycle in $G$ containing an odd number of twisted edges.

With the combinatorial definition of a graph embedding provided by an embedding scheme, the corresponding facial walks can be computed by the \textbf{face traversal algorithm}.
At each step of this algorithm we let $\kappa$ denote whether we follow the clockwise successor at the current vertex ($\kappa =1$) or the anticlockwise successor ($\kappa=-1$). Whenever the algorithm traverses an edge $e$ with $\lambda(e)=-1$, the sign of $\kappa$ changes.
To compute a facial walk of a graph $G$ with an embedding scheme $(\rho,\lambda)$, we begin by choosing an edge $e=\{v,w\}\in E(G)$ and setting $\kappa=1$. 
We then traverse $e$ from $v$ to $w$. If $\lambda(e)=1$, the walk proceeds with the clockwise successor of $e$ at $w$, namely $\rho_w(e)$. If $\lambda(e)=-1$, the walk instead follows the anticlockwise successor $\rho_w^{-1}(e)$ at $w$ and the value of $\kappa$ changes to $-1$.
We proceed analogously: we traverse edges by following the successor prescribed by $\kappa$ and $\rho$ until we encounter a twisted edge, at which point the sign of $\kappa$ changes, as illustrated in Figure~\ref{fig:twistedArcs}.
The facial walk is completed once we return to $v$ in such a way that the next edge to be traversed would be the
starting edge $e$ and with $\kappa=1$.
In this way, a single facial walk of the given embedding is obtained. To compute all facial walks of the embedding defined by $(\rho,\lambda)$, we then select another starting edge that has not yet been traversed twice by previously constructed facial walks.
Further details of this well-known algorithm can be found in \cite[Chapter 3.3.]{GraphsOnSurfaces}, for example.

\begin{figure}[ht!]
    \centering
    \begin{subfigure}{0.2\textwidth}
        \centering
        \begin{tikzpicture}[scale=1.]
            \tikzset{knoten/.style={circle,fill=black,inner sep=0.7mm}}
            \node [knoten] (a) at (0,0) {};
            \node [knoten] (b) at (0,-1) {};
            \node [knoten] (c) at ({1/sqrt(2)},{1/sqrt(2)}) {};
            \node [knoten] (d) at (-{1/sqrt(2)},{1/sqrt(2)}) {};
            \node [knoten] (e) at ({1/sqrt(2)},{-1/sqrt(2)-1}) {};
            \node [knoten] (f) at (-{1/sqrt(2)},{-1/sqrt(2)-1}) {};
    
            \draw[-,thick] (a) to (b);
            \draw[-,thick] (a) to (c);
            \draw[-,thick] (a) to (d);
            \draw[-,thick] (b) to (e);
            \draw[-,thick] (b) to (f);
    
            \node (E) at (0.3,-0.5) {$e$};
    
            \draw[orange,very thick,rounded corners] ({1/sqrt(2)+0.1},{1/sqrt(2)-0.1}) -- (0.1,-0.1) -- (0.1,-0.9) -- ({1/sqrt(2)+0.1},{-1/sqrt(2)-0.9});
            \draw[blue,very thick,rounded corners] ({-1/sqrt(2)-0.1},{1/sqrt(2)-0.1}) -- (-0.1,-0.1) -- (-0.1,-0.9) -- ({-1/sqrt(2)-0.1},{-1/sqrt(2)-0.9});
        \end{tikzpicture}
        \caption{}
        \label{subfig:untwistededge}
    \end{subfigure}
    \begin{subfigure}{0.2\textwidth}
        \centering
        \begin{tikzpicture}[scale=1.]
            \tikzset{knoten/.style={circle,fill=black,inner sep=0.7mm}}
            \node [knoten] (a) at (0,0) {};
            \node [knoten] (b) at (0,-1) {};
            \node [knoten] (c) at ({1/sqrt(2)},{1/sqrt(2)}) {};
            \node [knoten] (d) at (-{1/sqrt(2)},{1/sqrt(2)}) {};
            \node [knoten] (e) at ({1/sqrt(2)},{-1/sqrt(2)-1}) {};
            \node [knoten] (f) at (-{1/sqrt(2)},{-1/sqrt(2)-1}) {};
    
            \draw[-,very thick,red] (a) to (b);
            \draw[-,thick] (a) to (c);
            \draw[-,thick] (a) to (d);
            \draw[-,thick] (b) to (e);
            \draw[-,thick] (b) to (f);
    
            \node (E) at (0.3,-0.5) {$e$};
    
            \draw[orange,very thick,rounded corners] ({1/sqrt(2)+0.1},{1/sqrt(2)-0.1}) -- (0.1,-0.1) -- (0.1,-0.5) -- (-0.1,-0.5) -- (-0.1,-0.9) -- ({-1/sqrt(2)-0.1},{-1/sqrt(2)-0.9});
            \draw[blue,very thick,rounded corners] ({-1/sqrt(2)-0.1},{1/sqrt(2)-0.1}) -- (-0.1,-0.1) -- (-0.1,-0.5) -- (0.1,-0.5) -- (0.1,-0.9) -- ({1/sqrt(2)+0.1},{-1/sqrt(2)-0.9});
        \end{tikzpicture}
        \caption{}
        \label{subfig:twistededge}
    \end{subfigure}
    \caption{Segments of facial walks (blue and orange) for $e$ not twisted (a) and $e$ twisted (b).}
    \label{fig:twistedArcs}
\end{figure}

\subsection{Re-embeddings}\label{sec:reembedding}

In this section, we describe how an embedding of a $3$-connected cubic planar graph $G$ on a given non-spherical surface can be characterized via twisted edges. The described approach was first investigated by \cite{EnamiEmbeddings}.
Since a sphere is orientable, a planar embedding of $G$ is uniquely defined by a rotation system $\rho$. In fact, any drawing of $G$ on the plane implicitly defines such a rotation system.
An embedding $\beta$ of $G$ on a surface of higher genus can then be uniquely described by drawing $G$ on the plane and indicating which edges are twisted. In particular, we define a signature $\lambda$ by specifying a set $\mathcal{T}\subseteq E(G)$ with
$$\lambda: E(G)\rightarrow\{-1,1\},\,e\mapsto
    \begin{cases}
        -1 & \text{if } e\in\mathcal{T}\\
        1 & \text{otherwise. }
    \end{cases}
$$
We denote the embedding $\beta$ with the embedding scheme $(\rho,\lambda)$ as $\beta_\mathcal{T}(G)$ to emphasize that the unique planar embedding of $G$ provides the rotation system $\rho$ and that the edges in $\mathcal{T}$ define the signature $\lambda$. We refer to the embedding $\beta_\mathcal{T}(G)$ as a \textbf{re-embedding} of $G$.
The subgraph of $G^{\ast}$ induced by the edges dual to those in $\mathcal{T}$ is denoted by $H_\mathcal{T}$ and referred to as the \textbf{twisted subgraph} of $G^{\ast}$.
Moreover, a twisted subgraph $H_\mathcal{T}$ of $G^\ast$ determines the corresponding set of twisted edges of $G$ by taking the dual edges of $H_\mathcal{T}$.
Note that each set of twisted edges together with the rotation system $\rho$, defines an embedding on some surface. However, to obtain a re-embedding on a specific surface, it is essential to understand how to select the twisted edges appropriately.

\cite{EnamiEmbeddings} characterizes the sets of twisted edges that yield embeddings on the projective plane, the torus and the Klein bottle.
His approach defines these admissible sets of twisted edges by specific subgraphs of the dual graph, as summarized in the following theorem.
\begin{theorem}\cite[Theorem 1-3]{EnamiEmbeddings} \label{theorem:enami}
Let $G$ be a $3$-connected cubic planar graph. There exists a one-to-one correspondence between inequivalent embeddings of $G$ on
    \begin{itemize}
        \item[1)] the projective plane and subgraphs of $G^{\ast}$ that are isomorphic to $K_2$ or $K_4$.
        \item[2)] the torus and subgraphs of $G^{\ast}$ that are isomorphic to $K_{2,2,2}, K_{2,2m}$ or $K_{1,1,2m-1}$ for $m\geq 1$.
        \item[3)] the Klein bottle and subgraphs of $G^{\ast}$ that are isomorphic to $K_{2,2m-1}$ or $K_{1,1,2m}$ for $m\geq 1$, or one of the six graphs $A_1$ to $A_6$ shown in Figure~\ref{fig:SubgraphsKleinBottle} and~\ref{fig:SubgraphsKleinBottleGeneral}.
    \end{itemize}
\end{theorem}
\begin{figure}[ht!]
    \centering
    \begin{subfigure}{0.25\textwidth}
        \centering
        \begin{tikzpicture}[scale=1.8]
            \tikzset{knoten/.style={circle,fill=black,inner sep=0.7mm}}
            \node [knoten] (a) at (0,0) {};
            \node [knoten] (b) at (0,1) {};
            \node [knoten] (c) at (0.5,0) {};
            \node [knoten] (d) at (0.5,1) {};
            
            \draw[-,thick] (a) to (b);
            \draw[-,thick] (c) to (d);
            
        \end{tikzpicture}
        \caption{$A_1$}
    \end{subfigure}
    \begin{subfigure}{0.25\textwidth}
        \centering
        \begin{tikzpicture}[scale=1.8]
            \tikzset{knoten/.style={circle,fill=black,inner sep=0.7mm}}
            \node [knoten] (a) at (0.2,0) {};
            \node [knoten] (b) at (0.2,1) {};
            \node [knoten] (c) at (0.5,0) {};
            \node [knoten] (d) at (1.5,0) {};
            \node [knoten] (e) at (1,0.5) {};
            \node [knoten] (f) at (1,1) {};
            
            \draw[-,thick] (a) to (b);
            \draw[-,thick] (c) to (d);
            \draw[-,thick] (c) to (e);
            \draw[-,thick] (c) to (f);
            \draw[-,thick] (d) to (e);
            \draw[-,thick] (d) to (f);
            \draw[-,thick] (e) to (f);
            
        \end{tikzpicture}
        \caption{$A_2$}
    \end{subfigure}
    \begin{subfigure}{0.25\textwidth}
        \centering
        \begin{tikzpicture}[scale=1.8]
            \tikzset{knoten/.style={circle,fill=black,inner sep=0.7mm}}
            \node [knoten] (a) at (0,0) {};
            \node [knoten] (b) at (1,0) {};
            \node [knoten] (c) at (0.5,0.5) {};
            \node [knoten] (d) at (0.5,1) {};
            \node [knoten] (e) at (0.5,1.5) {};
            
            \draw[-,thick] (a) to (b);
            \draw[-,thick] (a) to (c);
            \draw[-,thick] (a) to (d);
            \draw[-,thick] (b) to (c);
            \draw[-,thick] (b) to (d);
            \draw[-,thick] (c) to (d);
            \draw[-,thick] (d) to (e);
          
        \end{tikzpicture}
        \caption{$A_4$}
    \end{subfigure}
    \caption{The three graphs $A_1$ (a), $A_2$ (b) and $A_4$ (c)}
    \label{fig:SubgraphsKleinBottleGeneral}
\end{figure}
 
This characterization yields a concrete approach for computing an embedding of a given $3$-connected cubic planar graph $G$ on a surface $S$ with non-negative Euler characteristic:
First, the planar embedding of $G$ must be computed to obtain the dual graph $G^\ast$. Then, depending on the surface $S$, we have to find a corresponding subgraph $H$ of $G^\ast$.
We define the dual edges of $H$ to be the twisted edges $\mathcal{T}\subseteq E(G)$ of the re-embedding $\beta_\mathcal{T}(G)$, where $H=H_\mathcal{T}$ is the twisted subgraph of $G^\ast$. With this, the facial walks of the re-embedding $\beta_\mathcal{T}(G)$ on the chosen surface can be computed by the face traversal algorithm. 
To compute all inequivalent embeddings of $G$ on a surface with non-negative Euler characteristic, we must first identify all relevant subgraphs of $G^\ast$ and subsequently translate each into a corresponding re-embedding $\beta_\mathcal{T}(G)$.

\begin{example}\label{ex:torus}
    As an example, we consider the graph $G$ with its planar embedding drawn in Figure~\ref{subfig:G}, which we would like to embed on a torus. For this we compute the dual graph $G^{\ast}$, drawn in Figure~\ref{subfig:dual}. The dual graph $G^\ast$ has a subgraph $H$ isomorphic to a $K_{2,2}$, which defines an embedding on a torus as stated in Theorem~\ref{theorem:enami}. The edges of $H$ are colored red in Figure~\ref{subfig:dual}.
    The rotation system $\rho$, defined by the drawing of $G$ on the plane, and the signature $\lambda$, defined by the dual edges of $H$, define an embedding scheme of $G$ on a torus. Since a torus is an orientable surface, there exists an equivalent embedding scheme such that all edges are untwisted \cite[Chapter 3.3]{GraphsOnSurfaces}. Thus, by local changes we compute an embedding scheme of $G$ equivalent to 
    $(\rho,\lambda)$ without twisted edges. The resulting embedding of $G$ on the torus is shown in Figure~\ref{subfig:embTorus}, where the edges of $H$ are still colored red.
\end{example}
\input{Figures/Example_Re-embedding}

\section{Characterizations of Strong Re-embeddings} \label{sec: strongreemb}
In this section, we consider strong embeddings of 3-connected cubic planar graphs on surfaces with non-negative Euler characteristic. As noted in the introduction, strong embeddings play a central role in the study of simplicial surfaces. For a given cubic graph $G$, multiple simplicial surfaces may exist whose face graphs are isomorphic to $G$. Note that computing all such surfaces is equivalent to enumerating all strong embeddings of $G$.

\subsection{Dual Facial Walks}
We begin by identifying the conditions under which a twisted subgraph yields a strong embedding of a 3-connected cubic planar graph, referred to as a \textbf{strong re-embedding}. 
The characterization of Theorem~\ref{theorem:strong} serves as the foundation for the proof of Theorem~\ref{theorem:main}.
However, as the following lemma demonstrates, certain choices of twisted edges never define strong re-embeddings.

\begin{lemma}\label{TwistedSimplicial}
    Let $G$ be a $3$-connected cubic planar graph and $\mathcal{T}\subseteq E(G)$. If a face of the planar embedding of $G$ contains exactly two twisted edges in its boundary that are adjacent in $G$, then the resulting re-embedding $\beta_\mathcal{T}(G)$ is not strong.
\end{lemma}
\begin{proof}
    Suppose $C$ is a face of $G$ embedded on the sphere whose boundary contains exactly two adjacent twisted edges $e_1,e_2\in \mathcal{T}$ with $e_1=\{v,w\}$ and $e_2=\{v,u\}$, as illustrated in Figure~\ref{RestrictionTwisted2}. Let $\rho$ be the rotation system of $\beta_\mathcal{T}(G)$ with $\rho_v(e_1)=e_2$. We start the face traversal algorithm by traversing $e_1$ from $v$ to $w$.
    Since $e_1$ is twisted, we consider the anticlockwise successor of $e_1$ at $w$, i.e.\ $\rho^{-1}_w(e_1)$. We continue the face traversal algorithm by traversing $C$ through the anticlockwise successors until we traverse $e_2$. Then we are back at $v$, but since $e_2$ is twisted, we do not return to $e_1$ and terminate, but continue with the other face containing $e_2$. Thus, the face traversal algorithm results in a closed walk which is not a cycle, as shown in Figure~\ref{RestrictionTwisted2}.
\end{proof}
\begin{figure}[ht!]
        \centering
        \begin{tikzpicture}[scale=2.2]
            \tikzset{knoten/.style={circle,fill=black,inner sep=0.7mm}}
            \node [knoten,label=above:$v$] (a) at (0,0) {};
            \node [knoten,label=above:$w$] (b) at (1,0) {};
            \node [knoten,label=below:$u$] (c) at (-0.5,-0.5) {};
            \node [knoten] (d) at (1.5,-0.5) {};
            
            \draw[-,red, very thick] (a) to (b);
            \draw[-,red, very thick] (a) to (c);
            \draw[-] (b) to (d);
            \draw[-] (c) to (-0.2,-0.8);
            \draw[-] (d) to (1.2,-0.8);
            \draw[-] (a) to (-0.2,0.2);
            \draw[-] (b) to (1.2,0.2);
            \draw[-] (c) to (-0.8,-0.5);
            \draw[-] (d) to (1.8,-0.5);
            \draw[loosely dotted] (-0.2,-0.8) to (1.2,-0.8);
            \draw[blue,very thick,rounded corners] (-0.25,0.15) -- (-0.1,0) -- (-0.3,-0.2) -- (-0.2,-0.3) -- (-0.4,-0.5)  -- (-0.2,-0.7) -- (1.2,-0.7) -- (1.4,-0.5) -- (1,-0.08) -- (0.5,-0.08) -- (0.5,0.08) -- (0.155,0.08);
            \node (C) at (0.5,-0.5) {$C$};
            \node (D) at (0.5,0.2) {$e_1$};
            \node (E) at (-0.5,-0.2) {$e_2$};
        \end{tikzpicture}
         \caption{A face that has exactly two adjacent twisted edges colored red contained in the boundary, together with its facial walk colored blue.}
        \label{RestrictionTwisted2}
    \end{figure}
 
Analogously, some of the twisted subgraphs mentioned in Theorem~\ref{theorem:enami} give rise to facial walks that are not cycles. We characterize the subgraphs that always give rise to strong re-embeddings and those that do not. To do this, we need a better understanding of the cases in which the resulting embeddings are not strong.

Let $G$ be a $3$-connected cubic planar graph embedded on the sphere with the rotation system $\rho$ and let $\lambda$ be a signature of $G$ that designates a subset $\mathcal{T}\subseteq E(G)$ as twisted. 
The embedding scheme $(\rho,\lambda)$ of $G$ induces an embedding scheme $(\rho^{\ast},\lambda^{\ast})$ of $G^{\ast}$, where $\rho^{\ast}$ denotes the rotation system defined by the planar embedding of $G^{\ast}$ induced by $\rho$, and $\lambda^{\ast}$ is defined such that all dual edges of $\mathcal{T}$, denoted by $\mathcal{T}^\ast$, are twisted. 
Moreover, we consider the restriction of $(\rho^{\ast},\lambda^{\ast})$ to $H_\mathcal{T}$, denoted by $(\rho^{\ast},\lambda^{\ast})\big|_{H_\mathcal{T}}$, obtained by simply omitting everything not contained in $H_\mathcal{T}$. This embedding scheme defines an embedding of $H_\mathcal{T}$, which, in analogy with the re‑embeddings of $G$, is denoted by $\beta _{\mathcal{T^{\mathnormal{\ast }}}}(H_{\mathcal{T}})$.
Note that in this embedding scheme of $H_\mathcal{T}$ the rotation system is uniquely determined by $\rho$ and that all edges are twisted. Consequently, the orientation in the facial walks alternates at every step.
The facial walks of $\beta_\mathcal{T}(G)$ and those of $\beta_{\mathcal{T}^\ast}(H_\mathcal{T})$ are closely related, as we will show in Lemma~\ref{lemma:dualCycle}.

The following definition identifies those edges of the dual graph that are not part of the twisted subgraph but whose position in the local rotation still influences the resulting facial walks. Based on this notion, we formulate a characterization that determines whether a twisted subgraph defines a strong embedding, independently of the underlying cubic planar graph.
Here, we denote the rotation system of $(\rho^{\ast},\lambda^{\ast})\big|_{H_\mathcal{T}}$ by $\rho'$ for simplicity.

\begin{definition}\label{def:dualcycles}
    Let $G$ be a $3$-connected cubic planar graph, $\mathcal{T}\subseteq E(G)$ and $\rho'$ the rotation system of the re-embedding $\beta_{\mathcal{T}^\ast}(H_\mathcal{T})$ determined by the unique planar embedding of $G$.
    Furthermore, let $C=(w_1,\dots,w_\ell)$ be a facial walk of $\beta_{\mathcal{T}^\ast}(H_\mathcal{T})$ with $e_i=\{w_{i-1},w_i\}\in E(H_\mathcal{T})$ for $i\in\{2,\dots,\ell\}$ and $e_1=\{w_1,w_\ell\}\in E(H_\mathcal{T})$.
    By definition the edges $e_1,\dots,e_\ell$ are twisted in the re-embedding $\beta_{\mathcal{T}^\ast}(H_\mathcal{T})$.
    First, consider the edge $e_i$ for $i$ odd. Thus, along the facial walk $C$ the edge $e_{i+1}$ is the clockwise successor under $\rho'$ of $e_i$ at $w_i$ with $e_{\ell+1}:=e_1$. Let $f_i^1,\dots, f_i^k\in E(G^\ast)\setminus E(H_\mathcal{T})$ denote the edges of $G^\ast$ that are incident to $w_i$ and lie between $e_i$ and $e_{i+1}$ in the clockwise ordering defined by $\rho'$ where $k$ depends on $i$. This means that we have $\rho'_{w_i}(e_i)=f_i^1$, $\rho'
    _{w_i}(f_i^j)=f_i^{j+1}$ for $j\in\{1,\dots,k-1\}$ and $\rho'_{w_i}(f^k_i)=e_{i+1}$, as shown in Figure~\ref{fig:visitedEdgesa}. We refer to the edges $f_i^1,\dots, f_i^k\in E(G^\ast)\setminus E(H_\mathcal{T})$ as the \emph{\textbf{visited edges}} of $C$ at $w_i$ and denote the tuple $(f_i^1,\dots,f_i^k)$ of visited edges at $w_i$ by $M_i$.
    The case for $e_i$ with even $i$ is analogous, replacing clockwise by anticlockwise and $\rho'$ by $(\rho')^{-1}$ with $(\rho')^{-1}=\{(\rho'_v)^{-1}\mid v\in V(H_\mathcal{T})\}$.
    With this we define the \emph{\textbf{dual facial walk}} of $H_\mathcal{T}$ in $G^{\ast}$, corresponding to $C$, as $(w_1^{M_1},\dots,w_\ell^{M_\ell})$, where the notation $w_i^{M_i}$ records the vertex $w_i$ of $C$ together with its visited edges $M_i$ for $i\in\{1,\dots,\ell\}$.
\end{definition}

For a better understanding of the definition, an illustration is given in Figure~\ref{fig:visitedEdges} showing the circular ordering around a vertex $w_i$ in $G^\ast$ with its visited edges of $C$ colored teal. In Figure~\ref{fig:visitedEdgesa}, the visited edges, i.e.\ the edges of $G^\ast$ that are not included in $H_\mathcal{T}$ but still influence the facial walks of $\beta_\mathcal{T}(G)$, are the edges lying between $e_i$ and $e_{i+1}$ in clockwise order. In Figure~\ref{fig:visitedEdgesb}, they are the edges lying between $e_i$ and $e_{i+1}$ in anticlockwise order.

\begin{figure}[ht!]
    \centering
    \begin{subfigure}{0.45\textwidth}
        \centering
        \begin{tikzpicture}[scale=1.8]
    \tikzset{knoten/.style={circle,fill=black,inner sep=0.7mm}}
    \node [knoten,label=$w_i$] (a) at (0,0) {};
    
    \node [knoten] (b) at (-1,0) {};
    \node [knoten] (c) at ({-sqrt(2)/2},{sqrt(2)/2}) {};
    \node [knoten] (d) at ({sqrt(2)/2},{sqrt(2)/2}) {};
    \node [knoten] (e) at (1,0) {};
    \node [knoten] (f) at ({sqrt(2)/2},{-sqrt(2)/2}) {};
    \node [knoten] (g) at ({-sqrt(2)/2},{-sqrt(2)/2}) {};
    
    \draw[-,very thick,red] (a) to (b);
    \draw[-,very thick,teal] (a) to (c);
    \draw[-,very thick,teal] (a) to (d);
    \draw[-,very thick,red] (a) to (e);
    \draw[-, thick] (a) to (f);
    \draw[-, thick] (a) to (g);

    \node (D) at (-0.5,-0.15) {$e_i$};
    \node (D) at (0.6,-0.15) {$e_{i+1}$};

    \node (D) at (-0.58,0.3) {$f_i^1$};
    \node (D) at (0.52,0.3) {$f_i^k$};

    \draw[ dotted,line width=1pt] (-0.2,{sqrt(2)/2}) to (0.2,{sqrt(2)/2});
    \draw[ dotted,line width=1pt] (-0.2,{-sqrt(2)/2}) to (0.2,{-sqrt(2)/2});

    \draw[->, bend left=40] (-0.25,0.5) to (0.25,0.5);

\end{tikzpicture}
        \caption{}
        \label{fig:visitedEdgesa}
    \end{subfigure}
    \begin{subfigure}{0.45\textwidth}
        \centering
        \begin{tikzpicture}[scale=1.8]
    \tikzset{knoten/.style={circle,fill=black,inner sep=0.7mm}}
    \node [knoten,label=$w_i$] (a) at (0,0) {};
    
    \node [knoten] (b) at (-1,0) {};
    \node [knoten] (c) at ({-sqrt(2)/2},{sqrt(2)/2}) {};
    \node [knoten] (d) at ({sqrt(2)/2},{sqrt(2)/2}) {};
    \node [knoten] (e) at (1,0) {};
    \node [knoten] (f) at ({sqrt(2)/2},{-sqrt(2)/2}) {};
    \node [knoten] (g) at ({-sqrt(2)/2},{-sqrt(2)/2}) {};
    
    \draw[-,very thick,red] (a) to (b);
    \draw[-, thick] (a) to (c);
    \draw[-, thick] (a) to (d);
    \draw[-,very thick,red] (a) to (e);
    \draw[-,very thick,teal] (a) to (f);
    \draw[-,very thick,teal] (a) to (g);

    \node (D) at (-0.5,0.1) {$e_i$};
    \node (D) at (0.6,0.1) {$e_{i+1}$};

    \node (D) at (-0.52,-0.3) {$f_i^1$};
    \node (D) at (0.55,-0.3) {$f_i^k$};

    \draw[ dotted,line width=1pt] (-0.2,{sqrt(2)/2}) to (0.2,{sqrt(2)/2});
    \draw[ dotted,line width=1pt] (-0.2,{-sqrt(2)/2}) to (0.2,{-sqrt(2)/2});

    \draw[->, bend right=40] (-0.25,-0.5) to (0.25,-0.5);

\end{tikzpicture}
        \caption{}
        \label{fig:visitedEdgesb}
    \end{subfigure}
    \caption{The local rotation at a vertex $w_i\in H_\mathcal{T}$ of the planar embedding of $G^\ast$ with $e_i,e_{i+1}\in E(H_\mathcal{T})$ and its visited edges $f_i^1,\dots,f_i^k$ of $C$ for $i$ odd (a) and $i$ even (b). The visited edges are shown in teal and the twisted edges in red.}
    \label{fig:visitedEdges}
\end{figure}

To proceed, we now establish a precise correspondence between the facial walks of $\beta_\mathcal{T}(G)$ and the dual facial walks of $H_\mathcal{T}$ in $G^\ast$.
For this, note first that a walk of a graph can be written as a sequence of edges. 
For the following definition we additionally use the notation $(E_1,e_1,E_2,e_2,\dots,E_\ell,e_\ell)$ of a walk in $G$, where $e_i\in E(G)$ and $E_i$ is a sequence of edges of $G$ forming a walk for $i\in\{1,\dots,\ell\}$.

\begin{definition}\label{def:corr}
    Let $G$ be a $3$-connected cubic planar graph with the facial cycles $F_1,\dots,F_m$ of its planar embedding and $\mathcal{T}\subseteq E(G)$. Moreover, let $\mathcal{F}(\beta_\mathcal{T}(G))$ denote the set of facial walks of $\beta_\mathcal{T}(G)$ containing twisted edges and $\mathcal{D}(H_\mathcal{T})$ the set of dual facial walks of $H_\mathcal{T}$ in $G^\ast$.
    Let $C$ be a facial walk of $\beta_\mathcal{T}(G)$ which can be written as $(E_1,e_1,E_2,e_2,\dots,E_\ell,e_\ell)$ such that $(e_i,E_{i+1},e_{i+1})$ is a subwalk of a facial cycle $F_{k_i}$ for all $i\in\{1,\dots,\ell\}$ with $e_{\ell+1}=e_1, E_{\ell+1}=E_1$ and $e_i\in\mathcal{T}$ for all $i\in\{1,\dots,\ell\}$. Thus, $\ell$ is even and the direction in which we traverse the faces of $F_{k_1},\dots,F_{k_\ell}$ alternates: for odd $i$ the face is traversed clockwise, whereas for even $i$ it is traversed anticlockwise.
    All dual edges of the edges in the subwalk $(e_i,E_{i+1},e_{i+1})$ are incident to the vertex $w_i$ in $H_\mathcal{T}$ that corresponds to the face $F_{k_i}$. However, only the dual edges of $e_i$ for $i\in\{1,\dots,\ell\}$ are actually contained in $H_\mathcal{T}$.
    We denote the sequence of the dual edges of the sequence $E_i$ by $E_i^\ast$ for $i\in\{1,\dots,\ell\}$.
    With this, we define $$\varphi:\mathcal{F}(\beta_\mathcal{T}(G))\rightarrow \mathcal{D}(H_\mathcal{T}),\,(E_1,e_1,E_2,e_2,\dots,E_\ell,e_\ell)\mapsto (w_1^{E_1^\ast},w_2^{E_2^\ast},\dots,w_\ell^{E_\ell^\ast}).$$
\end{definition}

The following statement shows that the map $\varphi$ describes a one-to-one correspondence between the facial walks of $\beta_\mathcal{T}(G)$ that contain twisted edges and the dual facial walks of $H_\mathcal{T}$ in $G^\ast$. This means that the dual facial walks can be used to uniquely determine the facial walks of $\beta_\mathcal{T}(G)$.

\begin{lemma}\label{lemma:dualCycle}
    Let $G$ be a $3$-connected cubic planar graph and $\mathcal{T}\subseteq E(G)$. Moreover, let $\mathcal{F}(\beta_\mathcal{T}(G))$ denote the set of facial walks of $\beta_\mathcal{T}(G)$ that contain twisted edges and $\mathcal{D}(H_\mathcal{T})$ the set of dual facial walks of $H_\mathcal{T}$ in $G^\ast$.
    The map $\varphi:\mathcal{F}(\beta_\mathcal{T}(G))\rightarrow \mathcal{D}(H_\mathcal{T})$ defined in Definition~\ref{def:corr} is well-defined and bijective.
\end{lemma}
\begin{proof}
    We use the notation introduced in Definition~\ref{def:corr}. Thus, $C=(E_1,e_1,E_2,e_2,\dots,E_\ell,e_\ell)$ is a facial walk of $\beta_\mathcal{T}(G)$ such that $(e_i,E_{i+1},e_{i+1})$ is a subwalk of a facial cycle $F_{k_i}$ of the planar embedding of $G$ for all $i\in\{1,\dots,\ell\}$ with $e_{\ell+1}=e_1$ and $E_{\ell+1}=E_1$. Moreover, $w_i$ is the vertex of $G^\ast$ that corresponds to $F_{k_i}$.
    First we show that $C^\ast:=\varphi(C)=\varphi((E_1,e_1,E_2,e_2,\dots,E_\ell,e_\ell))=(w_1^{E_1^\ast},w_2^{E_2^\ast},\dots,w_\ell^{E_\ell^\ast})$ is indeed a dual facial walk of $H_\mathcal{T}$ in $G^\ast$ with $E_1^\ast,\dots,E_\ell^\ast$ denoting the dual edge sequences of $E_1,\dots,E_\ell$. Since $C$ is a closed walk, it follows directly that $(w_1,w_2,\dots,w_\ell)$ is a closed walk in $H_\mathcal{T}$. 
    We recall that the facial cycles $F_{k_i}$ for $i\in\{1,\dots,\ell\}$ are traversed clockwise for $i$ odd and anticlockwise for $i$ even. Since these cycles correspond to the vertices $w_i$ of $H_\mathcal{T}$ for $i\in\{1,\dots,\ell\}$, traversing $e_i^\ast=\{w_i,w_{i+1}\}$ with $w_{\ell+1}=w_1$ reverse the ordering in which we consider the successor at $w_{i+1}$: for odd $i$ from clockwise to anticlockwise and for $i$ even the other way around. Thus, $(w_1,w_2,\dots,w_\ell)$ is a facial walk of the re-embedding $\beta_{\mathcal{T}^\ast}(H_\mathcal{T})$.
    Moreover, by definition, the dual sequence of $E_i$ for $i\in\{1,\dots,\ell\}$ is precisely the sequence of visited edges of $(w_1,w_2,\dots,w_\ell)$ at $w_i$. Thus, $C^\ast$ is a dual facial walk of $H_\mathcal{T}$ in $G^\ast$.
    Note that a facial walk $(E_1,e_1,E_2,e_2,\dots,E_\ell,e_\ell)$ of $\beta_\mathcal{T}(G)$ may be cyclically shifted or reversed but it still represents the same facial walk of $\beta_\mathcal{T}(G)$. The same holds for the dual facial walks of $H_\mathcal{T}$ where additionally the visited edges are reversed if the corresponding facial walk of $\beta_{\mathcal{T}^\ast}(H_\mathcal{T})$ is reversed. Thus, $\varphi$ is well-defined.

    It remains to show that $\varphi$ is bijective. We do this by defining $$\psi:\mathcal{D}(H_\mathcal{T})\rightarrow \mathcal{F}(\beta_\mathcal{T}(G)),\, (v_1^{M_1},v_2^{M_2},\dots,v_\ell^{M_\ell})\mapsto (M_1^\ast,f_1^\ast,M_2^\ast,f_2^\ast,\dots,M_\ell^\ast,f_\ell^\ast)$$ where $f_i^\ast$ denotes the dual edge of $f_i:=\{v_i,v_{i+1}\}$ where $v_{\ell+1}=v_1$ and $M_i^\ast$ is the dual edge sequence of $M_i$ for $i\in\{1,\dots,\ell\}$. We claim that $\psi$ is the inverse map of $\varphi$.
    The closed walk $(M_1^\ast,f_1^\ast,M_2^\ast,f_2^\ast,\dots,\\ M_\ell^\ast,f_\ell^\ast)$ is indeed a facial walk of $\beta_\mathcal{T}(G)$ since each edge $f_i$ for $i\in\{1,\dots,\ell\}$ is a twisted edge of $\beta_\mathcal{T}(G)$ and the ordering by traversing $f_i$ always changes by the definition of the dual facial walks.
    Consequently, by the same reasoning as for $\varphi$, the map $\psi$ is well‑defined.
    It remains to show that $\varphi$ and $\psi$ are inverse to each other. For this let $C=(E_1,e_1,E_2,e_2,\dots,E_\ell,e_\ell)$ be a facial walk of $\beta_\mathcal{T}(G)$ as in Definition~\ref{def:corr}. We compute
    \begin{align*}
        \psi(\varphi(C))&=\psi(\varphi((E_1,e_1,E_2,e_2,\dots,E_\ell,e_\ell)))=\psi((w_1^{E_1^\ast},w_2^{E_2^\ast},\dots,w_\ell^{E_\ell^\ast}))\\&=(E_1,e_1,E_2,e_2,\dots,E_\ell,e_\ell)=C
    \end{align*} since $e_i$ is the dual edge of $\{w_i,w_{i+1}\}$ for $i\in\{1,\dots,\ell\}$ and $w_{\ell+1}=w_1$.
    Similarly,
    $$\varphi(\psi((v_1^{M_1},v_2^{M_2},\dots,v_\ell^{M_\ell})))=\varphi((M_1^\ast,f_1^\ast,M_2^\ast,f_2^\ast,\dots,M_\ell^\ast,f_\ell^\ast))=(v_1^{M_1},v_2^{M_2},\dots,v_\ell^{M_\ell})$$
    by the definition of $\psi$ and $\varphi$.
    Thus, $\psi$ is inverse to $\varphi$ and so $\varphi$ is bijective.
\end{proof}

This statement shows that the maps $\varphi$ and $\psi$ establish a one-to-one correspondence between the facial walks of $\beta_\mathcal{T}(G)$  that contain twisted edges and the dual facial walks of $\mathcal{T}^\ast$ in $G^\ast$. We will later use this correspondence to identify conditions on the dual facial walks of $H_\mathcal{T}$ in $G^\ast$ that ensure that the facial walks of $\beta_\mathcal{T}(G)$ are cycles. Before doing so, we present an example that illustrates the concept of the dual facial walks of $H_\mathcal{T}$ in $G^\ast$ and their correspondence to the facial walks of the re-embedding $\beta_\mathcal{T}(G)$.

\begin{example}\label{ex:dualcycle}
    As an example, consider the $3$-connected cubic planar graph $G$ with its planar embedding shown in Figure~\ref{subfig:cubic6} and let $H_\mathcal{T}\cong K_{2,2}$ be defined by $\mathcal{T}=\{\,\{1,2\},\{1,3\},\{4,6\},\{5,6\}\,\}\subseteq E(G)$, colored red. Let the faces of the planar embedding of $G$ be labeled as follows:
    $$a=(2,3,4,5),\,b=(1,2,5,6),\,c=(1,2,3),\,d=(1,3,4,6)\text{ and } e=(4,5,6).$$
    In Figure~\ref{subfig:dualcubic6}, the planar embedding of $G^{\ast}$, induced by the planar embedding of $G$, is drawn with the labeling inherited from the face labels of $G$. The red edges indicate the subgraph $H_\mathcal{T}$. Consider the facial cycle $C=(c,d,e,b)$ of $\beta_{\mathcal{T}^\ast}(H_\mathcal{T})$ starting in clockwise orientation. The twisted edges of $C$ are colored red, the visited edges of $C$ at all vertices are colored teal and the facial cycle $C$ is colored blue in Figure~\ref{subfig:dualcubic6}. Thus, the dual facial walk of $H_\mathcal{T}$ in $G^\ast$ corresponding to $C$ is $(c^{(\{a,c\})},d^{(\{b,d\})},e^{(\{a,e\})},b^{(\{b,d\})})$. Translating all the twisted and visited edges to their dual edges we obtain $(2,3,1,6,4,5,6,1)$ as a facial walk of $\beta_\mathcal{T}(G)$, see Figure~\ref{subfig:cubic6_facial}. Thus, the re-embedding $\beta_\mathcal{T}(G)$ is not strong.
    \begin{figure}[ht!]
    \centering
    \begin{subfigure}{0.38\textwidth}
        \centering
        \begin{tikzpicture}[scale=2.]
            \tikzset{knoten/.style={circle,fill=black,inner sep=0.7mm}}
            \node [knoten,label=left:1] (a) at (-1,0) {};
            \node [knoten,label=above:2] (b) at (-0.5,0.5) {};
            \node [knoten,label=below:3] (c) at (-0.5,-0.5) {};
            \node [knoten,label=below:4] (d) at (0.5,-0.5) {};
            \node [knoten,label=above:5] (e) at (0.5,0.5) {};
            \node [knoten,label=right:6] (f) at (1,0) {};
            \draw[-,red,very thick] (a) to (b);
            \draw[-,red,very thick] (a) to (c);
            \draw[-,thick] (b) to (c);
            \draw[-,thick] (b) to (e);
            \draw[-,thick] (c) to (d);
            \draw[-,thick] (d) to (e);
            \draw[-,red,very thick] (d) to (f);
            \draw[-,red,very thick] (e) to (f);
            \draw[-,thick] (-1,-0.05) arc (-180:0:1cm);;

            \node [cyan] (A) at (0,0) {$a$};
            \node [cyan] (B) at (0,0.8) {$b$};
            \node [cyan] (C) at (-0.7,0) {$c$};
            \node [cyan] (D) at (0,-0.8) {$d$};
            \node [cyan] (E) at (0.7,0) {$e$};
            
        \end{tikzpicture}
        \caption{}
        \label{subfig:cubic6}
    \end{subfigure}
    \begin{subfigure}{0.25\textwidth}
        \centering
        \begin{tikzpicture}[scale=2.8]
            \tikzset{knoten/.style={circle,fill=black,inner sep=0.7mm}}
            \node [knoten,label=above right:$a$] (a) at (0,0) {};
            \node [knoten,label=$b$] (b) at (0,0.5) {};
            \node [knoten,label=$c$] (c) at (-0.45,0) {};
            \node [knoten,label=below:$d$] (d) at (0,-0.5) {};
            \node [knoten,label=$e$] (e) at (0.45,0) {};
            \draw[-,thick] (a) to (b);
            \draw[-,teal,ultra thick] (a) to (c);
            \draw[-,thick] (a) to (d);
            \draw[-,teal,ultra thick] (a) to (e);
            \draw[-,red,very thick] (b) to (c);
            \draw[-,red,very thick] (b) to (e);
            \draw[-,red,very thick] (c) to (d);
            \draw[-,red,very thick] (d) to (e);
            \draw[-,teal,ultra thick] (-0.05,0.5) arc (90:270:0.5cm);;

            \draw[blue,very thick,rounded corners] (0,0.58) -- (-0.2,0.35) -- (-0.2,0.2) -- (-0.4,0) -- (-0.2,-0.2) -- (-0.2,-0.35) -- (0,-0.58) -- (0.2,-0.35) -- (0.2,-0.2) -- (0.4,0) -- (0.2,0.2) -- (0.2,0.35) -- cycle;
        \end{tikzpicture}
        \caption{}
        \label{subfig:dualcubic6}
    \end{subfigure}
    \begin{subfigure}{0.35\textwidth}
        \centering
        \begin{tikzpicture}[scale=2.]
            \tikzset{knoten/.style={circle,fill=black,inner sep=0.7mm}}
            \node [knoten,label=left:1] (a) at (-1,0) {};
            \node [knoten,label=above:2] (b) at (-0.5,0.5) {};
            \node [knoten,label=below:3] (c) at (-0.5,-0.5) {};
            \node [knoten,label=below:4] (d) at (0.5,-0.5) {};
            \node [knoten,label=above:5] (e) at (0.5,0.5) {};
            \node [knoten,label=right:6] (f) at (1,0) {};
            \draw[-,red,very thick] (a) to (b);
            \draw[-,red,very thick] (a) to (c);
            \draw[-,thick] (b) to (c);
            \draw[-,thick] (b) to (e);
            \draw[-,thick] (c) to (d);
            \draw[-,thick] (d) to (e);
            \draw[-,red,very thick] (d) to (f);
            \draw[-,red,very thick] (e) to (f);
            \draw[-,thick] (-1,-0.05) arc (-180:0:1cm);;

            \draw[blue,very thick,rounded corners] (0.95,-0.1) arc (-5:-160:0.95cm) -- (-0.95,-0.1) -- (-0.85,-0.25) -- (-0.75,-0.15) -- (-0.55,-0.4) -- (-0.55,0.35) -- (-0.75,0.15) -- (-0.8,0.3) -- (-1.08,0) arc (-180:-5:1.1cm) -- (1.08,0) -- (0.8,0.3) -- (0.75,0.15) -- (0.55,0.35) -- (0.55,-0.4) -- (0.75,-0.15) -- (0.85,-0.25) -- cycle;
            
        \end{tikzpicture}
        \caption{}
        \label{subfig:cubic6_facial}
    \end{subfigure}
    \caption{A planar embedded graph $G$ (a), its dual graph $G^{\ast}$ (b) with the twisted edges of $\beta_{\mathcal{T}^\ast}(H_\mathcal{T})$ colored red and
    a facial walk of $\beta_{\mathcal{T}^\ast}(H_\mathcal{T})$ colored blue with the visited edges colored teal and $G$ with the resulting facial walk of $\beta_\mathcal{T}(G)$ colored in blue (c).}
\end{figure}
\end{example}
 
The goal is to use dual facial walks to characterize whether a given embedding scheme defines a strong embedding or not. Before addressing this, we first establish a result concerning facial walks of $\beta_{\mathcal{T}^\ast}(H_\mathcal{T})$ that are not edge-simple.
A closed walk is called \textbf{edge-simple} if all of its edges are distinct. In particular, any edge-simple closed walk in a cubic graph is necessarily a cycle.

\begin{proposition}\label{lemma:edgesimple}
    Let $G$ be a $3$-connected cubic planar graph and $\mathcal{T}\subseteq E(G)$. If a facial walk of $\beta_{\mathcal{T}^\ast}(H_\mathcal{T})$ is not edge-simple, then the re-embedding $\beta_\mathcal{T}(G)$ is not strong.
\end{proposition}
\begin{proof}
     Let $\Tilde{C}$ be a facial walk of $\beta_{\mathcal{T}^\ast}(H_\mathcal{T})$ that traverses the edge $e\in H_\mathcal{T}$ more than once. Then the dual facial walk of $H_\mathcal{T}$ in $G^\ast$ that corresponds to $\Tilde{C}$ contains the edge $e$ also more than once, but not among the visited edges. By the bijection $\varphi$ between the facial walks of $\beta_\mathcal{T}(G)$ and the dual facial walks of $H_\mathcal{T}$ in $G^\ast$, as established in Lemma~\ref{lemma:dualCycle}, the dual edge of $e$ occurs twice in the corresponding facial walk $C$ of $\beta_\mathcal{T}(G)$. Thus, $C$ is not a cycle, and so the re-embedding $\beta_\mathcal{T} (G)$ is not strong.
\end{proof}

\begin{corollary}\label{cor:bridge}
     Let $G$ be a $3$-connected cubic planar graph and $\mathcal{T}\subseteq E(G)$. If $H_\mathcal{T}$ has a bridge, then the re-embedding $\beta_\mathcal{T} (G)$ is not strong.
\end{corollary}
\begin{proof}
    Because $H_\mathcal{T}$ has a bridge, the facial walks of $\beta_{\mathcal{T}^\ast}(H_\mathcal{T})$ cannot be edge-simple. Hence, by Proposition~\ref{lemma:edgesimple}, the re-embedding $\beta_\mathcal{T} (G)$ is not strong.
\end{proof}

To compute the facial walks of $\beta_\mathcal{T}(G)$ without applying the face traversal algorithm to the corresponding embedding scheme $(\rho,\lambda)$ of $G$, we proceed as follows: 
We first compute all facial walks of $\beta_{\mathcal{T}^\ast}(H_\mathcal{T})$ using the face traversal algorithm for $H_\mathcal{T}$ and $(\rho^{\ast},\lambda^{\ast})\big|_{H_\mathcal{T}}$, as defined above.
Since we are only interested in edge-simple facial walks by Proposition~\ref{lemma:edgesimple}, we exclude the cases where the facial walks of $\beta _{\mathcal{T^{\mathnormal{\ast }}}}(H_{\mathcal{T}})$ are not edge-simple.
For each edge-simple facial walk, we compute the dual facial walk as described in Definition~\ref{def:dualcycles}. From these dual walks, we derive the sequence of twisted and visited edges, which in turn yield the facial walks of $\beta_\mathcal{T}(G)$ that contain twisted edges, as shown in Lemma~\ref{lemma:dualCycle}.
This method now serves as the basis for characterizing which types of subgraphs fail to produce strong re-embeddings.

\begin{lemma} \label{lemma:notStrong}
    Let $G$ be a $3$-connected cubic planar graph and $\mathcal{T}\subseteq E(G)$. If $H_\mathcal{T}$ has a vertex of degree one or is isomorphic to $K_{1,1,m}$ for $m\geq 1$, the re-embedding $\beta_\mathcal{T}(G)$ is not strong.
\end{lemma}
\begin{proof}
    If $H_\mathcal{T}$ has a vertex of degree one, $H_\mathcal{T}$ has a bridge. Thus, the re-embedding $\beta_\mathcal{T}(G)$ is not strong by Corollary~\ref{cor:bridge}.
    
    Let $H_\mathcal{T}$ be isomorphic to $K_{1,1,m}$  for $m\geq 1$ with the partition sets $\{w_1\},\{w_2\},\{w_3,\dots,w_{m+2}\}$. Without loss of generality, we may assume that $K_{1,1,m}$ is planar embedded as shown in Figure~\ref{K11m}, in which the partition sets are distinguished by color.
    We compute the facial walks of $\beta_{\mathcal{T}^\ast}(H_\mathcal{T})$ with the face traversal algorithm, starting at vertex $w_1$ with the edge $\{w_1,w_2\}$ in clockwise orientation. Then we obtain the facial walk $(w_1,w_2,w_{m+2},w_1,w_2,w_3)$. This facial walk contains the edge $\{w_1,w_2\}$ twice, so it is not a cycle and therefore the re-embedding $\beta_\mathcal{T}(G)$ is not strong.
\end{proof}
\begin{figure}[ht!]
        \centering
        \begin{tikzpicture}[scale=1.8]
            \tikzset{knoten/.style={circle,fill=black,inner sep=0.7mm}}
            \node [knoten,label=left:$w_3$,red] (a) at (-0.5,0) {};
            \node [knoten,label=right:$w_4$,red] (b) at (0,0) {};
            \node [knoten,label=left:\small $w_{m+1}$,red] (c) at (3,0) {};
            \node [knoten,label=right:$w_{m+2}$,red] (d) at (3.5,0) {};
            \node [knoten,label=$w_1$,blue] (e) at (1.5,1) {};
            \node [knoten,label=below:$w_2$,teal] (f) at (1.5,-1) {};
            \node (u) at (1.5,0) {$\dots$};
            
            \draw[-,thick] (e) to (a);
            \draw[-,thick] (e) to (b);
            \draw[-,thick] (e) to (c);
            \draw[-,thick] (e) to (d);
            \draw[-,thick] (f) to (a);
            \draw[-,thick] (f) to (b);
            \draw[-,thick] (f) to (c);
            \draw[-,thick] (f) to (d);
    
            \draw[thick] (1.4,1) arc
                [
                    start angle=90,
                    end angle=270,
                    x radius=2.5cm,
                    y radius =1cm
                ] ;
        \end{tikzpicture}
        \caption{Planar embedding of $H_\mathcal{T}\cong K_{1,1,m}$}
        \label{K11m}
    \end{figure}
 
The following theorem provides a criterion
for determining whether a given twisted subgraph yields a strong re-embedding. We will apply this criterion to each of Enami's subgraphs to assess whether they define a strong re-embedding.
\begin{theorem}\label{theorem:strong}
    Let $G$ be a $3$-connected cubic planar graph and $H_\mathcal{T}$ the twisted subgraph of $G^{\ast}$ for $\mathcal{T}\subseteq E(G)$ such that the facial walks of $\beta_{\mathcal{T}^\ast}(H_\mathcal{T})$ are edge-simple.
    The re-embedding $\beta_\mathcal{T}(G)$ is strong if and only if for every dual facial walk $$C^\ast=\left(w_1^{M_1},\dots,w_{\ell}^{M_{\ell}} \right)$$ of $H_\mathcal{T}$ in $G^{\ast}$, the visited edges $M_1,\dots, M_\ell$ are pairwise disjoint.
\end{theorem}
\begin{proof}
     We prove both directions separately.
    \begin{itemize}
        \item["$\Rightarrow$"] Assume that the edge sets $M_a$ and $M_b$ of $C^\ast$ are not disjoint for $a,b\in\{1,\dots,\ell\}$. Let $e$ be an edge with $e\in M_a\cap M_b$, i.e.\ we visit $e$ once at the vertex $w_a$ and once at $w_b$ in $C^\ast$.
        The dual facial walk $C^\ast$ of $H_\mathcal{T}$ can be translated to a facial walk $C$ of $\beta_\mathcal{T}(G)$ by $\psi$ the inverse map of $\varphi$ as described in Definition~\ref{def:corr} and the proof of Lemma~\ref{lemma:dualCycle}. Thus, the dual edge of $e$ is traversed twice by the facial walk $C$: once in the face of the planar embedding of $G$ corresponding to $w_a$ and once in the face corresponding to $w_b$.
        This means that the re-embedding $\beta_\mathcal{T}(G)$ is not strong.
        \item["$\Leftarrow$"]
        Assume that the re-embedding $\beta_\mathcal{T}(G)$ is not strong.
        The faces of the planar embedding of $G$ without a twisted edge are facial cycles of $\beta_\mathcal{T}(G)$. Since $G$ is $3$-connected, we only have to consider the facial walks of $\beta_\mathcal{T}(G)$ with at least one twisted edge. These are precisely the facial walks of $\beta_\mathcal{T}(G)$ obtained from the dual facial walks of $H_\mathcal{T}$ in $G^\ast$ as shown in Lemma~\ref{lemma:dualCycle}. 
        Since $\beta_\mathcal{T}(G)$ is not strong, there exists a facial walk $C$ of $\beta_\mathcal{T}(G)$ that is not a cycle.
        Thus, $C$ must traverse an edge $e\in E(G)\setminus\mathcal{T}$ twice, since we assumed that the facial walks of $\beta_{\mathcal{T}^\ast}(H_\mathcal{T})$ are edge-simple.
        Hence, in the dual facial walk $C^\ast$ of $H_\mathcal{T}$ in $G^\ast$ corresponding to $C$ exist two vertices $w_a$ and $w_b$ at which the dual edge of $e$ occurs among the visited edges $M_a$ and $M_b$.
        So the sets $M_j$ for $j\in\{1,\dots,\ell\}$ are not pairwise disjoint.
    \end{itemize}
\end{proof}

\begin{remark}\label{remark:faces}
    Within the framework of the above theorem, it follows directly that if the visited edges of a given facial walk of $\beta_{\mathcal{T}^\ast}(H_\mathcal{T})$ all lie in distinct faces of the planar embedding of $H_\mathcal{T}$ induced by the planar embedding of $G$, then the corresponding facial walk of $\beta_\mathcal{T}(G)$ is a cycle. Consequently, if this condition is satisfied for all dual facial walks of $H_\mathcal{T}$ in $G^\ast$, then the re-embedding $\beta_\mathcal{T}(G)$ is strong. Note that the opposite direction is not true in general.
\end{remark}

The simplification of Theorem \ref{theorem:strong} described in the remark is useful in situations where the graph under consideration is large and it suffices, as a first step, to determine in which faces the visited edges lie.

\begin{example}
    In Example~\ref{ex:torus} and \ref{ex:dualcycle} the considered twisted subgraph $H_\mathcal{T}$ is isomorphic to $K_{2,2}$. Note that $\beta_{\mathcal{T}^\ast}(H_\mathcal{T})$ for $H_\mathcal{T}\cong K_{2,2}$ has two facial walks both containing the same edges but with different visited edges.
    In Figure~\ref{fig:dualCycle} the planar embedding of $G^\ast$ of Example~\ref{ex:torus} is drawn where the visited edges of the two facial walks of $\beta_{\mathcal{T}^\ast}(H_\mathcal{T})$ are highlighted in teal and green, respectively. Both facial walks of $\beta_{\mathcal{T}^\ast}(H_\mathcal{T})$ have visited edges that lie in the same face of $H_\mathcal{T}$. However, all visited edges are incident to only one vertex of $H_\mathcal{T}$, i.e.\ the visited edges are all disjoint and so the re-embedding $\beta_\mathcal{T}(G)$ is strong.
    In Example~\ref{ex:dualcycle} the visited edges of the facial walk illustrated in Figure~\ref{subfig:dualcubic6} at the vertices $b$ and $d$ intersect in the edge $\{b,d\}$. Thus, the re-embedding described in Example~\ref{ex:dualcycle} is not strong.
\end{example}
\begin{figure}[ht!]
    \centering
        \begin{tikzpicture}[scale=2.2]  
            \tikzset{knoten/.style={circle,fill=black,inner sep=0.7mm}}

            \draw[-,ultra thick, teal] (1,1) arc (90:-90:1cm);
            
            \node [knoten, label={[cyan]:$a$}] (a) at (0,0) {};
            \node [knoten, label={[cyan]below left:$b$}] (b) at (0.75,0) {};
            \node [knoten, label={[cyan]$f$}] (c) at (1,1) {};
            \node [knoten, label={[cyan]below:$g$}] (d) at (1,-1) {};
            \node [knoten, label={[cyan]left:$e$}] (e) at (1.3,-0.35) {};
            \node [knoten, label={[cyan]left:$c$}] (f) at (1.25,0.4) {};
            \node [knoten, label={[cyan]$d$}] (g) at (1.85,0) {};
    
            \draw[-,ultra thick,green] (a) to (b);
            \draw[-,thick] (a) to (c);
            \draw[-,ultra thick, teal] (a) to (d);
            \draw[-,ultra thick,green] (b) to (c);
            \draw[-,very thick,red] (b) to (d);
            \draw[-,ultra thick,teal] (b) to (e);
            \draw[-,very thick,red] (b) to (f);
            \draw[-,ultra thick, teal] (c) to (f);
            \draw[-,ultra thick, green] (c) to (g);
            \draw[-,ultra thick, green] (d) to (e);
            \draw[-,very thick,red] (d) to (g);
            \draw[-,ultra thick, green] (e) to (f);
            \draw[-,ultra thick, teal] (e) to (g);
            \draw[-,very thick,red] (f) to (g);
        \end{tikzpicture}
    \caption{The planar embedding of $G^{\ast}$ with $G$ as drawn in Figure~\ref{subfig:G} where the edges of $H_\mathcal{T}\cong K_{2,2}$ are colored red and the disjoint visited edges of the two facial walks of the re-embedding $\beta_{\mathcal{T}^\ast}(H_\mathcal{T})$ are colored teal and green.}
    \label{fig:dualCycle}
\end{figure}

\subsection{On the Projective Plane}
In this section, we consider strong re-embeddings of $3$-connected cubic planar graphs on the projective plane, i.e.\ on a non-orientable surface of genus 1. The theorem of Enami states that we need to find subgraphs of the dual graph that are isomorphic to $K_2$ or $K_4$ to obtain general embeddings on the projective plane. Note that Lemma~\ref{lemma:notStrong} already establishes that twisted subgraphs which are isomorphic to $K_2$ do not lead to strong re-embeddings. Thus the question arises, whether a twisted subgraph isomorphic to $K_4$ always leads to a strong re-embedding.

\begin{lemma}\label{lemma:K4strong}
    Let $G$ be a $3$-connected cubic planar graph and $\mathcal{T}\subseteq E(G)$. If $H_\mathcal{T}$ is isomorphic to $K_4$, then the re-embedding $\beta_\mathcal{T}(G)$ is strong.
\end{lemma}
\begin{proof}
    We first compute the facial walks of $\beta_{\mathcal{T}^\ast}(H_\mathcal{T})$, which yields three facial cycles of length four, see Figure~\ref{fig:K4proj}.
    Thus, let $C^\ast=(w_1^{M_1},w_2^{M_2},w_3^{M_3},w_4^{M_4})$ be a dual facial walk of $H_\mathcal{T}$ in $G^{\ast}$. Suppose $M_a$ and $M_b$ are not disjoint with $e\in M_a\cap M_b$ for $a,b\in\{1,\dots,4\}$ and $a\neq b$. This means that the vertices $w_a$ and $w_b$ of $H_\mathcal{T}$ are connected in $G^{\ast}$ by the untwisted edge $e$. However, $w_a$ and $w_b$ are also adjacent by a twisted edge, since $H_\mathcal{T}$ is a complete graph. This implies that the two faces $w_a$ and $w_b$ of the planar embedding of $G$ would intersect in two edges, which is not possible since $G$ is $3$-connected. So the statement follows with Theorem~\ref{theorem:strong}.
\end{proof}
\begin{figure}[ht!]
    \centering
    \begin{tikzpicture}[scale=1]
    \tikzset{knoten/.style={circle,fill=black,inner sep=0.7mm}}
    \node [knoten] (a) at (0,0) {};
    \node [knoten] (b) at (4,0) {};
    \node [knoten] (c) at (2,4) {};
    \node [knoten] (d) at (2,1.5) {};
    \draw[-,very thick] (a) to (b);
    \draw[-,very thick] (c) to (a);
    \draw[-,very thick] (d) to (c);
    \draw[-,very thick] (a) to (d);
    \draw[-,very thick] (b) to (c);
    \draw[-,very thick] (b) to (d);

    \draw[blue,very thick,rounded corners] (-0.05,0.15) -- (0.85,2) -- (1.17,2) -- (1.87,3.5) -- (1.87,2.5) -- (2.12,2.5) -- (2.12,1.55) -- (3,0.9) -- (3,0.6) -- (3.62,0.15) -- (2,0.15) -- (2,-0.15) -- (-0.2,-0.15) -- cycle;

    \draw[green!60!black,very thick,rounded corners] (0.22,0.25) -- (1,1.8) -- (0.95,2.2) -- (2,4.25) -- (3.05,2.2) -- (3,1.8) -- (3.78,0.25) -- (3.1,0.8) -- (2.8,0.8) -- (2,1.35) -- (1.2,0.8) -- (0.9,0.8) -- cycle;

    \draw[orange,very thick,rounded corners] (2,0.15) -- (2,-0.15) -- (4.2,-0.15) -- (3.25,1.8) -- (2.75,2.2) -- (2.13,3.5) -- (2.13,2.5) -- (1.87,2.5) -- (1.88,1.55) -- (1.1,1.) -- (1,0.6) -- (0.38,0.15) -- cycle;
\end{tikzpicture}
    \caption{The facial cycles of $\beta_{\mathcal{T}^\ast}(H_\mathcal{T})$ for $H_\mathcal{T}\cong K_4$.}
    \label{fig:K4proj}
\end{figure}

By the characterization of \cite{EnamiEmbeddings} together with Lemma~\ref{lemma:notStrong} and Lemma~\ref{lemma:K4strong} we directly obtain the first part of Theorem~\ref{theorem:main}.
Having established a complete characterization of strong re-embeddings on the projective plane, along with an implementation of the corresponding algorithm by \cite{simplicialsurfacegap}, we now investigate the number of $3$-connected cubic planar graphs that admit a strong embedding on the projective plane.

Let $\mathcal{G}_n$ denote the set of $3$-connected cubic planar graphs with $n$ vertices and let $\mathcal{P}_n$ denote the set of graphs in $\mathcal{G}_n$ that have a strong embedding on the projective plane. Table~\ref{table:projective} shows that there are few graphs in $\mathcal{G}_n$ for $n\geq 8$ which do not have a strong embedding on the projective plane. This means that their dual graphs have no subgraph isomorphic to $K_4$.

\begin{table}[ht!]
    \centering
    \begin{tabular}[ht!]{|c|c|c|c|c|c|c|c|c|c|c|c|}
    \hline
    \textbf{$n$} & \textbf{ 4} &  \textbf{6}& \textbf{8} &\textbf{ 10} &\textbf{ 12} & \textbf{14}&\textbf{16}&\textbf{18}&\textbf{20}\\
    \hline
     $\mid\mathcal{G}_n\mid$ & 1& 1& 2& 5& 14& 50& 233& 1249& 7595\\
     \hline
     $\mid\mathcal{P}_n\mid$ & 1& 1& 1& 4& 12& 45& 222& 1219& 7485\\
     \hline
    \end{tabular}
    \caption{Number of graphs in $\mathcal{G}_n$ for $n\in\{4,6,8,10,12,14,16,18,20\}$ with strong embeddings on the projective plane.}
    \label{table:projective}
\end{table}
 
It is clear that each vertex of degree three in the dual graph defines a subgraph isomorphic to $K_4$. Thus, all dual graphs of $\mathcal{G}_n$ that do not have a strong embedding on the projective plane have minimum degree at least four. Thus, the number of $\mid\mathcal{G}_n\mid{-}\mid\mathcal{P}_n\mid$ is always less than or equal to the number of plane triangulations with $\frac{n}{2}+2$ vertices and minimum degree at least four.
Note that \cite{Dillencourt} enumerates the number of plane triangulations for a given number of vertices with minimum degree at least four. Their counts align with the numbers shown in Table~\ref{table:projective}.

\subsection{On the Torus}
In this section, we consider strong re-embeddings of $3$-connected cubic planar graphs on the torus, i.e.\ on an orientable surface of genus 1. The theorem of Enami states that we need to find subgraphs of the dual graph that are isomorphic to $K_{2,2,2}, K_{2,2m}$ or $K_{1,1,2m-1}$ with $m\geq 1$ to obtain general embeddings on the torus.
By Lemma~\ref{lemma:notStrong} we already know that twisted subgraphs which are isomorphic to $K_{1,1,2m-1}$ for $m\geq 1$ do not lead to strong re-embeddings. Thus, the question is whether a twisted subgraph which is isomorphic to $K_{2,2,2}$ or $K_{2,2m}$ for $m\geq 1$ always leads to a strong re-embedding.

\begin{lemma}\label{lemma:K222strong}
    Let $G$ be a $3$-connected cubic planar graph and $\mathcal{T}\subseteq E(G)$. If $H_\mathcal{T}$ is isomorphic to $K_{2,2,2}$, then the re-embedding $\beta_\mathcal{T}(G)$ is strong.
\end{lemma}
\begin{proof}
    The four facial walks of the re-embedding $\beta_{\mathcal{T}^\ast}(H_\mathcal{T})$ are edge-simple as depicted in Figure~\ref{fig:K222} and of length six.
    Let $C^\ast=(w_1^{M_1},\dots,w_6^{M_6})$ be a dual facial walk of $H_\mathcal{T}$ in $G^{\ast}$. Assume, towards a contradiction, that the edge sets $M_a$ and $M_b$ for $a,b\in\{1,\dots,6\}$ are not disjoint, so that there exists an edge $e\in M_a\cap M_b$ with $a\neq b$. This assumption implies that $w_a$ and $w_b$ are adjacent in $G^{\ast}$ by the untwisted edge $e$. Consider the following two cases:
    \begin{enumerate}
        \item Suppose $w_a$ and $w_b$ are in different partition sets of $H_\mathcal{T}$. Then $w_a$ and $w_b$ are additionally connected by an edge of $H_\mathcal{T}$, which contradicts the fact that two faces of a planar embedded $3$-connected cubic graph can intersect in at most one edge.
        \item Suppose $w_a$ and $w_b$ are in the same partition set of $H_\mathcal{T}$. Since $H_\mathcal{T}\cong K_{2,2,2}$ is planar and $3$-connected, $H_\mathcal{T}$ is uniquely embeddable on the sphere. Figure~\ref{fig:K222} shows such an embedding, where the vertices of the different partition sets are highlighted in red, blue and green. Two vertices of the same partition set of $H_\mathcal{T}$ cannot be adjacent in $G^{\ast}$ by an untwisted edge, otherwise we could not embed $G^{\ast}$ on the sphere. So we get a contradiction to the fact that $w_a$ and $w_b$ are both incident to $e$.
    \end{enumerate}
    Thus, the edge sets $M_a$ and $M_b$ must be disjoint. Since the dual facial walk and the edge sets were chosen arbitrary, we have shown that the visited edges are disjoint for all dual facial walks and the result follows by Theorem~\ref{theorem:strong}.
\end{proof}
\begin{figure}[ht!]
    \centering
    \begin{tikzpicture}[scale=1.2]
            \tikzset{knoten/.style={circle,fill=black,inner sep=0.6mm}}
            \node [knoten,red!70!black] (a) at (0,0) {};
            \node [knoten,green!70!black] (b) at (0,2) {};
            \node [knoten,blue!70!black] (c) at (-2,0) {};
            \node [knoten,green!70!black] (d) at (0,-2) {};
            \node [knoten,blue!70!black] (e) at (2,0) {};
            \node [knoten,red!70!black] (f) at (4,0) {};
            \draw[-,thick] (a) to (b);
            \draw[-,thick] (a) to (c);
            \draw[-,thick] (a) to (d);
            \draw[-,thick] (a) to (e);
            \draw[-,thick] (b) to (c);
            \draw[-,thick] (b) to (e);
            \draw[-,thick] (c) to (d);
            \draw[-,thick] (d) to (e);

            \draw[-,thick] (f) to (b);
            \draw[-,thick] (f) to (d);
            \draw[-,thick] (f) to (e);
            
            \draw[-,thick] (-2,0.04) arc (180:0:3cm);;

            \draw[orange,very thick,rounded corners] (-1.8,0.1) -- (-1,0.1) -- (-1,-0.1) -- (-0.1,-0.1) -- (-0.1,-1) -- (0.1,-1) -- (0.1,-1.8) -- (0.95,-0.95) -- (1.1,-1) -- (2,-0.1) -- (3,-0.1) -- (3,0.1) -- (3.6,0.1) -- (2,0.9) -- (2,1.1) -- (0,2.1) -- (-1,1.1) -- (-1,0.9) -- cycle;

            \draw[purple,very thick,rounded corners] (-1.8,-0.1) -- (-1,-0.1) -- (-1,0.1) -- (-0.1,0.1) -- (-0.1,1) -- (0.1,1) -- (0.1,1.8) -- (0.9,0.95) -- (1.1,1) -- (2,0.1) -- (3,0.1) -- (3,-0.1) -- (3.6,-0.1) -- (2,-0.9) -- (2,-1.1) -- (0,-2.1) -- (-1,-1.1) -- (-1,-0.9) -- cycle; 

            \draw[blue!40!white,very thick,rounded corners] (0.1,0.1) -- (1,0.1) -- (1,-0.1) -- (1.8,-0.1) -- (1,-0.9) -- (1.1,-1.05) -- (0.4,-1.7) -- (2,-0.9) -- (2,-1.1) -- (4.1,-0.1) arc (-1:87:3.2cm) -- (1,2.95) arc (90:175:2.9cm) -- (-1,1.1) -- (-0.9,0.95) -- (-0.1,1.8) -- (-0.1,1) -- (0.1,1) -- cycle;

            \draw[green!40!black,very thick,rounded corners] (0.1,-0.1) -- (1,-0.1) -- (1,0.1) -- (1.8,0.1) -- (0.9,0.95) -- (1.1,1) -- (0.4,1.7) -- (2,0.9) -- (2,1.1) -- (3.9,0.15) arc (2:90:2.9cm) -- (1,3.1) arc (90:180:3.1cm) -- (-1.05,-1.05) -- (-0.95,-0.95) -- (-0.1,-1.8) -- (-0.1,-1) -- (0.1,-1) -- cycle;
        \end{tikzpicture}
    \caption{All facial walks of $\beta_{\mathcal{T}^\ast}(H_\mathcal{T})$ with $H_\mathcal{T}\cong K_{2,2,2}$, where vertices from the same partition set are assigned the same color.}
    \label{fig:K222}
\end{figure}
 
So twisted subgraphs which are isomorphic to $K_{2,2,2}$ always lead to strong re-embeddings. In contrast, for twisted subgraphs isomorphic to $K_{2,2m}$ an additional property is needed. To establish this, we examine more closely the configuration of the visited edges at a vertex.
Recall that the Jordan Curve Theorem~\cite[]{JordanCurve} states that each cycle $C$ of a bridgeless planar graph divides the plane into the interior of $C$, bounded by $C$ and written as $int(C)$, and the exterior of $C$, written as $ext(C)$.

Let $G$ again be a $3$-connected cubic planar graph, $H_\mathcal{T}$ a twisted subgraph of $G^\ast$ and $\Tilde{C}$ a facial cycle of $\beta_{\mathcal{T}^\ast}(H_\mathcal{T})$. We consider $\Tilde{C}$ as a cycle in the planar embedding of $H_\mathcal{T}$ induced by the planar embedding of $G$. By the Jordan Curve Theorem the visited edges at a vertex $v$ of $\Tilde{C}$ are either all in $int(\Tilde{C})$ or all in $ext(\Tilde{C})$, depending on whether we consider the clockwise or anticlockwise successor at $v$.
\begin{lemma}\label{lemma:K2mstrong}
    Let $G$ be a $3$-connected cubic planar graph and $\mathcal{T}\subseteq E(G)$ such that $H_\mathcal{T}\cong K_{2,m}$ for $m\geq 2$. The re-embedding $\beta_\mathcal{T}(G)$ is strong if and only if the vertices of the partition set of size two of $H_\mathcal{T}$ are not adjacent in $G^{\ast}$. For $m=2$, guaranteeing a strong embedding requires that no edges of $G^{\ast }$ occur within either of the two partition sets.
\end{lemma}
\begin{proof}
    Let $H_\mathcal{T}$ consists of the vertices $w_1,\dots w_{m+2}$ where $w_1$ and $w_2$ are the vertices of the partition set of size $2$ and $w_3,\dots,w_{m+2}$ are the vertices of the partition set of size $m$.
    \begin{itemize}
        \item["$\Rightarrow$"] Let the two vertices $w_1$ and $w_2$ of the partition set of $H_\mathcal{T}$ of size two be connected by the edge $e\in E(G^{\ast})$. Then, without loss of generality, $H_\mathcal{T}\cup \{e\}$ is embedded on the plane as in Figure~\ref{subfig:K2mNotStrong}.
        \begin{figure}[ht!]
    \centering
    \begin{tikzpicture}[scale=1]
        \tikzset{knoten/.style={circle,fill=black,inner sep=0.7mm}}
        \node [knoten,label=right:$w_{m+2}$,blue] (a) at (4,0) {};
        \node [knoten,label=$w_1$,red] (b) at (2,1.5) {};
        \node [knoten,label=left:$w_3$,blue] (c) at (0,0) {};
        \node [knoten,label=below:$w_2$,red] (d) at (2,-1.5) {};
        \node (u) at (2,0) {$\dots$};
        
        \draw[-,thick] (a) to (b);
        \draw[-,thick] (b) to (c);
        \draw[-,thick] (c) to (d);
        \draw[-,thick] (d) to (a);

        \draw[-,thick] (b) to (2,0.7);
        \draw[-,thick] (b) to (2.5,0.7);
        \draw[-,thick] (b) to (1.5,0.7);

        \draw[-,thick] (d) to (2,-0.7);
        \draw[-,thick] (d) to (2.5,-0.7);
        \draw[-,thick] (d) to (1.5,-0.7);

        \draw[-,thick] (1.9,1.5) arc
                [
                    start angle=90,
                    end angle=270,
                    x radius=3cm,
                    y radius =1.5cm
                ] ;
    \end{tikzpicture}
    \caption{Planar embedding of $H_\mathcal{T}\cup \bigl\{ \{w_1,w_2\} \bigr\}$ with $H_\mathcal{T}\cong K_{2,m}$}
    \label{subfig:K2mNotStrong}
\end{figure}
        Consider the facial walk $C=(w_1,w_{m+2},w_2,w_3)$ in $\beta_{\mathcal{T}^\ast}(H_\mathcal{T})$. The visited edges at $w_1$ and $w_2$ of $C$ lie in $ext(C)$. Thus, the edge $\{w_1,w_2\}$ is a visited edge at $w_1$ and $w_2$ and so the re-embedding cannot be strong by Theorem~\ref{theorem:strong}.
        \item["$\Leftarrow$"] The dual graph $G^{\ast}$ contains a subgraph isomorphic to $K_{2,m}$, as drawn in Figure~\ref{fig:K2m}. The vertices of the partition set of size two, called $A\subseteq V(H_\mathcal{T})$, are colored red and the vertices of the partition set of size $m$, called $B\subseteq V(H_\mathcal{T})$, are colored blue.
        
        \begin{figure}[ht!]
    \centering
    \begin{tikzpicture}[scale=1]
            \tikzset{knoten/.style={circle,fill=black,inner sep=0.7mm}}
            \node [knoten,blue] (a) at (0,0) {};
            \node [knoten,red,label=$w_1$] (b) at (2,1.5) {};
            \node [knoten,blue] (c) at (1,0) {};
            \node [knoten,red,label=below:$w_2$] (d) at (2,-1.5) {};
            \node [knoten,blue] (e) at (4,0) {};
            \node (u) at (3,0) {$\dots$};
            
            \draw[-,thick] (a) to (b);
            \draw[-,thick] (b) to (c);
            \draw[-,thick] (c) to (d);
            \draw[-,thick] (d) to (a);
            \draw[-,thick] (b) to (e);
            \draw[-,thick] (d) to (e);

            \draw[-,thick] (b) to (2,0.7);
            \draw[-,thick] (b) to (2.5,0.7);

            \draw[-,thick] (d) to (2,-0.7);
            \draw[-,thick] (d) to (2.5,-0.7);
        \end{tikzpicture}
    \caption{Planar embedding of $K_{2,m}$ for $m\geq 2$}
    \label{fig:K2m}
\end{figure}
        
        The re-embedding $\beta_{\mathcal{T}^\ast}(H_\mathcal{T})$ has $m$ facial cycles of length four, all of which are also facial walks of the planar embedding of $H_\mathcal{T}$. Such a facial walk $C$ of $\beta_{\mathcal{T}^\ast}(H_\mathcal{T})$ contains the vertices $w_1$ and $w_2$ of $A$ and two vertices of $B$, called $w_3$ and $w_4$. We consider the planar embedding of $H_\mathcal{T}$ induced by the planar embedding of $G$, which is drawn in Figure~\ref{fig:K2m}, so $w_3$ and $w_4$ must be embedded next to each other, see Figure~\ref{subfig:case2}, or one is the leftmost vertex and the other is the rightmost vertex, see Figure~\ref{subfig:case1}.
        Thus, the facial walk $C=(w_1,w_3,w_2,w_4)$ of $\beta_{\mathcal{T}^\ast}(H_\mathcal{T})$ can be an unbounded or bounded face of the planar embedding of $H_\mathcal{T}$.
        Let $C^\ast=(w_1^{M_1},w_3^{M_3},w_2^{M_2},w_4^{M_4})$ be the dual facial walk of $H_\mathcal{T}$ in $G^{\ast}$ corresponding to $C$ as described in Definition~\ref{def:corr}.
        \begin{itemize}
            \item[1.] Let $C=(w_1,w_3,w_2,w_4)$ be the unbounded face of the planar embedding of $H_\mathcal{T}$. Accordingly, the facial walk $C$ is depicted in blue in Figure~\ref{subfig:case1}. We start at $w_1$ with the edge $\{w_1,w_3\}$ in clockwise orientation. This means that the edges of $M_1$ lie in the unbounded face of $H_\mathcal{T}$. Now we have to consider the anticlockwise successor of $\{w_1,w_3\}$ at $w_3$, which is $\{w_2,w_3\}$. Here the edges we are visiting, i.e.\ $M_3$, are in $int(C)$. Continuing this way, we see that the edges of $M_2$ lie in the unbounded face of $H_\mathcal{T}$ and that the edges of $M_4$ are in $int(C)$. Thus, according to Theorem~\ref{theorem:strong} the re-embedding is strong if $M_1\cap M_2=\emptyset$ and $M_3\cap M_4=\emptyset$. If $M_3$ and $M_4$ would have a common edge, $m$ must be equal to $2$ because the common edge must be in $int(C)$. Since we assumed that the vertices of partition sets of size two are not adjacent, both pairs of edge sets are disjoint.
            \item[2.] Let $C=(w_1,w_3,w_2,w_4)$ be a bounded face of $H_\mathcal{T}$. Accordingly, the facial walk $C$ is depicted in blue in Figure~\ref{subfig:case2}. Analogous to the first case, we obtain that $M_1$ and $M_2$ are in $int(C)$. Moreover, $M_3$ and $M_4$ are in $ext(C)$. Thus, according to Theorem~\ref{theorem:strong} the re-embedding is strong if $M_1\cap M_2=\emptyset$ and $M_3\cap M_4=\emptyset$.
            If $M_3$ and $M_4$ would have a common edge, $m$ must be equal to $2$ because the common edge must be in $ext(C)$. Since we assumed that the vertices of partition sets of size two are not adjacent, both pairs of edge sets are disjoint.
        \end{itemize}
    \end{itemize}
\end{proof}
\begin{figure}[ht!]
    \centering
    \begin{subfigure}{.45\textwidth}
        \centering
        \begin{tikzpicture}[scale=1.2]
            \tikzset{knoten/.style={circle,fill=black,inner sep=0.7mm}}
            \node [knoten,label=right:$w_3$] (a) at (4,0) {};
            \node [knoten,label=$w_1$] (b) at (2,1.5) {};
            \node [knoten,label=left:$w_4$] (c) at (0,0) {};
            \node [knoten,label=below:$w_2$] (d) at (2,-1.5) {};
            \node (u) at (2,0) {$\dots$};
            
            \draw[-,thick] (a) to (b);
            \draw[-,thick] (b) to (c);
            \draw[-,thick] (c) to (d);
            \draw[-,thick] (d) to (a);

            \draw[-,thick] (b) to (2,0.7);
            \draw[-,thick] (b) to (2.5,0.7);
            \draw[-,thick] (b) to (1.5,0.7);

            \draw[-,thick] (d) to (2,-0.7);
            \draw[-,thick] (d) to (2.5,-0.7);
            \draw[-,thick] (d) to (1.5,-0.7);

            \draw[blue,very thick,rounded corners] (2,1.7) -- (3,1) -- (2.8,0.75) -- (3.8,0) -- (2.8,-0.75) -- (3,-1) -- (2,-1.7) -- (1,-1) -- (1.2,-0.75) -- (0.2,0) -- (1.2,0.75) -- (1,1) -- cycle;
        \end{tikzpicture}
        \caption{}
        \label{subfig:case1}
    \end{subfigure}
    \begin{subfigure}{.45\textwidth}
        \centering
        \begin{tikzpicture}[scale=1.2]
            \tikzset{knoten/.style={circle,fill=black,inner sep=0.7mm}}
            \node [knoten,label=left:$w_3$] (a) at (0,0) {};
            \node [knoten,label=$w_1$] (b) at (2,1.5) {};
            \node [knoten,label=right:$w_4$] (c) at (1,0) {};
            \node [knoten,label=below:$w_2$] (d) at (2,-1.5) {};
            \node [knoten] (e) at (4,0) {};
            \node (u) at (3,0) {$\dots$};
            
            \draw[-,thick] (a) to (b);
            \draw[-,thick] (b) to (c);
            \draw[-,thick] (c) to (d);
            \draw[-,thick] (d) to (a);
            \draw[-,thick] (b) to (e);
            \draw[-,thick] (d) to (e);

            \draw[-,thick] (b) to (2,0.7);
            \draw[-,thick] (b) to (2.5,0.7);

            \draw[-,thick] (d) to (2,-0.7);
            \draw[-,thick] (d) to (2.5,-0.7);

            \draw[blue,very thick,rounded corners] (1.7,1.2) -- (1.4,0.8) -- (1.5,0.5) -- (1.15,0) -- (1.5,-0.5) -- (1.4,-0.8) -- (1.7,-1.2) -- (1.1,-0.7) -- (0.6,-0.7) -- (-0.2,0) -- (0.6,0.7) -- (1.1,0.7) -- cycle;
        \end{tikzpicture}
        \caption{}
        \label{subfig:case2}
    \end{subfigure}
    \caption{Planar embeddings of $H_\mathcal{T}\cong K_{2,m}$ with $(w_1,w_3,w_2,w_4)$ a face of $H_\mathcal{T}$ and the corresponding facial walk of $\beta_{\mathcal{T}^\ast}(H_\mathcal{T})$ colored in blue.}
    \label{fig:K2m2}
\end{figure}

By the characterization of \cite{EnamiEmbeddings} together with Lemma~\ref{lemma:notStrong}, Lemma~\ref{lemma:K222strong} and Lemma~\ref{lemma:K2mstrong} we directly obtain the second part of Theorem~\ref{theorem:main}.
Having established a complete characterization of strong re-embeddings on the torus, along with an implementation of the corresponding algorithm by~\cite{simplicialsurfacegap}, we now investigate the number of $3$-connected cubic planar graphs that admit a strong embedding on the torus.

Let $\mathcal{G}_n$ again denote the set of $3$-connected cubic planar graphs with $n$ vertices. Moreover, $\mathcal{R}_n$ denotes the set of graphs in $\mathcal{G}_n$ that have strong embeddings on the torus. Table~\ref{table:torus} shows that there are several graphs in $\mathcal{G}_n$, which do not have a strong embedding on the torus.

\begin{table}[ht!]
    \centering
    \begin{tabular}[ht!]{|c|c|c|c|c|c|c|c|c|c|c|c|}
\hline
\textbf{$n$} & \textbf{ 4} &  \textbf{6}& \textbf{8} &\textbf{ 10} &\textbf{ 12} & \textbf{14}&\textbf{16}&\textbf{18}&\textbf{20}\\
\hline
 $\mid\mathcal{G}_n\mid$ & 1& 1& 2& 5& 14& 50& 233& 1249& 7595\\
 \hline
 $\mid\mathcal{R}_n\mid$ & 0& 0& 1& 2& 7& 26& 140& 815& 5484\\
 \hline
\end{tabular}
\caption{Number of graphs in $\mathcal{G}_n$ for $n\in\{4,6,8,10,12,14,16,18,20\}$ with strong embeddings on the torus.}
\label{table:torus}
\end{table}

\subsection{On the Klein Bottle}
In this section we consider strong re-embeddings of $3$-connected cubic planar graphs on the Klein bottle, i.e.\ on a non-orientable surface of genus 2. The theorem of Enami states that we need to find subgraphs of the dual graph that are isomorphic to the graphs $A_1,\dots A_6$, $K_{2,2m-1}$ or $K_{1,1,2m}$ with $m\geq 1$ to obtain general re-embeddings on the Klein bottle.
By Lemma~\ref{lemma:notStrong} we already know that twisted subgraphs which are isomorphic to $K_{1,1,2m}$ do not lead to strong re-embeddings. In addition, by Corollary~\ref{cor:bridge} we know that $K_{2,1}$ also does not yield a strong re-embedding.
Moreover, Lemma~\ref{lemma:K2mstrong} implies that considering the subgraphs of $G^{\ast}$ which are isomorphic to $K_{2,2m-1}$ for $m\geq 2$, where the vertices of the partition set of size two are not adjacent in $G^{\ast}$ is enough.
To get a complete characterization for strong re-embeddings on the Klein bottle we need to take a look at the graphs $A_1,\dots A_6$, see Figure~\ref{fig:SubgraphsKleinBottle} and \ref{fig:SubgraphsKleinBottleGeneral}. Each of the graphs $A_1$, $A_2$, and $A_4$ contains a bridge and therefore does not define a strong embedding by Corollary~\ref{cor:bridge}.
\begin{lemma}\label{lemma:Bstrong}
     Let $G$ be a $3$-connected cubic planar graph and $\mathcal{T}\subseteq E(G)$. If $H_\mathcal{T}$ is isomorphic to $A_3, A_5$ or $A_6$, then the re-embedding $\beta_\mathcal{T}(G)$ is strong.
\end{lemma}
\begin{proof}
    The graph $A_3$ is equal to the disjoint union of two copies of $K_4$.
    Since the facial walks of these two components of $K_4$ are disjoint and edge-simple, Theorem~\ref{theorem:strong} and Lemma \ref{lemma:K4strong} implies directly that the re-embedding $\beta_\mathcal{T}(G)$ is strong for $H_\mathcal{T}\cong A_3$.
    
    The graph $A_5$ consists of two copies of the $K_4$ joined at a single vertex $v$. In the embedding $\beta _{\mathcal{T^{\mathnormal{\ast }}}}(H_{\mathcal{T}})$, where $H_{\mathcal{T}}\cong A_5$, all but one facial walk contain vertices from only one of the two $K_4$ components.
    By Lemma~\ref{lemma:K4strong}, the visited edges of these facial walks are disjoint. The remaining facial walk, illustrated in Figure~\ref{fig:A5}, traverses $v$ twice. However, the two sets of visited edges at $v$ must be disjoint since otherwise a loop in $G^\ast$ would occur. All other visited edges lie in distinct faces and are therefore disjoint by Remark~\ref{remark:faces}.
    Furthermore, all facial walks of $\beta _{\mathcal{T^{\mathnormal{\ast }}}}(H_{\mathcal{T}})$ are edge-simple and therefore Theorem~\ref{theorem:strong} implies that the re-embedding $\beta_\mathcal{T}(G)$ is strong.

    For $H_\mathcal{T}\cong A_6$ the facial walks of $\beta_{\mathcal{T}^\ast}(H_\mathcal{T})$ are edge-simple as depicted in Figure~\ref{fig:A6}.
    It can be easily seen that all visited edges of the facial walks of $\beta_{\mathcal{T}^\ast}(H_\mathcal{T})$ lie in distinct faces of $A_6$, and therefore, by Remark~\ref{remark:faces} and Theorem~\ref{theorem:strong}, the re-embedding $\beta _{\mathcal{T}}(G)$ is strong.
\end{proof}

\begin{figure}[ht!]
    \centering
    \begin{subfigure}{0.43\textwidth}
        \centering
        \raisebox{22pt}{\begin{tikzpicture}[scale=3.]
    \tikzset{knoten/.style={circle,fill=black,inner sep=0.7mm}}
    \node [knoten] (a) at (0,0) {};
    \node [knoten] (b) at (0,1) {};
    \node [knoten] (c) at (0.5,0.5) {};
    \node [knoten] (d) at (1,0.5) {};
    \node [knoten] (e) at (1.5,0.5) {};
    \node [knoten] (f) at (2,1) {};
    \node [knoten] (g) at (2,0) {};
    
    \draw[-,thick] (a) to (b);
    \draw[-,thick] (a) to (c);
    \draw[-,thick] (a) to (d);
    \draw[-,thick] (b) to (c);
    \draw[-,thick] (b) to (d);
    \draw[-,thick] (c) to (d);
    
    \draw[-,thick] (e) to (f);
    \draw[-,thick] (e) to (g);
    \draw[-,thick] (e) to (d);
    \draw[-,thick] (f) to (g);
    \draw[-,thick] (f) to (d);
    \draw[-,thick] (g) to (d);

    \draw[blue,very thick,rounded corners] (1,0.55) -- (1.4,0.75) -- (1.5,0.7) -- (1.85,0.88) -- (1.65,0.7) -- (1.74,0.7) -- (1.55,0.5) -- (1.74,0.3) -- (1.65,0.3) -- (1.85,0.12) -- (1.5,0.3) -- (1.4,0.25) -- (1,0.45) -- (0.6,0.25) -- (0.5,0.3) -- (0.15,0.12) -- (0.35,0.3) -- (0.26,0.3) -- (0.45,0.5) -- (0.26,0.7) -- (0.35,0.7) -- (0.15,0.88) -- (0.5,0.7) -- (0.6,0.75) -- cycle;
\end{tikzpicture}}
        \caption{}
        \label{fig:A5}
    \end{subfigure}
    \begin{subfigure}{0.56\textwidth}
        \centering
        \begin{tikzpicture}[scale=3.8]
    \tikzset{knoten/.style={circle,fill=black,inner sep=0.7mm}}
    \node [knoten] (a) at (0,0) {};
    \node [knoten] (b) at (0.5,0) {};
    \node [knoten] (c) at (1,0.5) {};
    \node [knoten] (d) at (1,-0.5) {};
    \node [knoten] (e) at (1.5,0) {};
    \node [knoten] (f) at (2,0) {};
    
    \draw[-,thick] (a) to (b);
    \draw[-,thick] (a) to (c);
    \draw[-,thick] (a) to (d);
    \draw[-,thick] (b) to (c);
    \draw[-,thick] (b) to (d);
    \draw[-,thick] (c) to (e);
    \draw[-,thick] (c) to (f);
    \draw[-,thick] (d) to (e);
    \draw[-,thick] (d) to (f);
    \draw[-,thick] (e) to (f);

    \draw[blue,very thick,rounded corners] (1,0.45) -- (0.8,0.25) -- (0.65,0.2) -- (0.5,0.03) -- (0.3,0.03) -- (0.26,-0.03) -- (0.1,-0.03) -- (0.47,-0.21) --  (0.5,-0.3) -- (1,-0.55) -- (1.5,-0.3) -- (1.53,-0.21) -- (1.9,-0.03) -- (1.74,-0.03) -- (1.7,0.03) -- (1.5,0.03) -- (1.35,0.2) -- (1.2,0.25) -- cycle;

    \draw[purple,very thick,rounded corners] (1,0.55) -- (0.5,0.3) -- (0.48,0.22) -- (0.1,0.03) -- (0.25,0.03) -- (0.3,-0.03) -- (0.5,-0.03) -- (0.66,-0.2) -- (0.8,-0.25) -- (1,-0.45) -- (1.2,-0.25) -- (1.34,-0.2) -- (1.5,-0.03) -- (1.7,-0.03) -- (1.75,0.03) -- (1.9,0.03) -- (1.55,0.2) -- (1.5,0.3) -- cycle;

    \draw[orange,very thick,rounded corners] (0.55,0) -- (0.77,0.22) -- (0.7,0.23) -- (0.87,0.4) -- (0.5,0.2) -- (0.48,0.28) -- (-0.1,0) -- (0.48,-0.28) -- (0.5,-0.2) -- (0.87,-0.4) -- (0.7,-0.23) -- (0.77,-0.22) -- cycle;

    \draw[green!60!black,very thick,rounded corners] (1.45,0) -- (1.23,0.22) -- (1.3,0.23) -- (1.13,0.4) -- (1.5,0.2) -- (1.52,0.28) -- (2.1,0) -- (1.52,-0.28) -- (1.5,-0.2) -- (1.13,-0.4) -- (1.3,-0.23) -- (1.23,-0.22) -- cycle;
\end{tikzpicture}
        \caption{}
        \label{fig:A6}
    \end{subfigure}
    \caption{A facial walk of $\beta_{\mathcal{T}^\ast}(H_\mathcal{T})$ with $H_\mathcal{T}\cong A_5$ (a) and all facial walks of $\beta_{\mathcal{T}^\ast}(H_\mathcal{T})$ with $H_\mathcal{T}\cong A_6$.}
\end{figure}

By the characterization of \cite{EnamiEmbeddings} together with Lemma~\ref{lemma:notStrong}, Lemma~\ref{lemma:K2mstrong} and Lemma~\ref{lemma:Bstrong} we directly obtain the last part of Theorem~\ref{theorem:main}.
Having established a complete characterization of strong re-embeddings on the Klein bottle, along with an implementation of the corresponding algorithm by~\cite{simplicialsurfacegap}, we now investigate the number of $3$-connected cubic planar graphs that admit a strong embedding on the Klein bottle.

Let $\mathcal{G}_n$ again denote the set of $3$-connected cubic planar graphs with $n$ vertices. Moreover, $\mathcal{K}_n$ denotes the set of graphs in $\mathcal{G}_n$ that have strong embeddings on the Klein bottle. Table~\ref{table:Kleinbottle} shows that all graphs in $\mathcal{G}_n$ with $n<20$ have a strong embedding on the Klein bottle. Furthermore, there is only one graph in $\mathcal{G}_{20}$ that does not have a strong embedding on the Klein bottle. This is the dodecahedral graph, which is the dual graph of the icosahedral graph.

\begin{table}[ht!]
    \centering
    \begin{tabular}[ht!]{|c|c|c|c|c|c|c|c|c|c|c|c|}
\hline
\textbf{$n$} & \textbf{ 4} &  \textbf{6}& \textbf{8} &\textbf{ 10} &\textbf{ 12} & \textbf{14}&\textbf{16}&\textbf{18}&\textbf{20}\\
\hline
 $\mid\mathcal{G}_n\mid$ & 1& 1& 2& 5& 14& 50& 233& 1249& 7595\\
 \hline
 $\mid\mathcal{K}_n\mid$ & 0& 1& 2& 5& 14& 50& 233& 1249& 7594\\
 \hline
\end{tabular}
\caption{Number of graphs in $\mathcal{G}_n$ for $n\in\{4,6,8,10,12,14,16,18,20\}$ with strong embeddings on the Klein bottle.}
\label{table:Kleinbottle}
\end{table}

\section{Existence of Strong Re-embeddings}\label{sec:exstrong}
\cite{EnamiEmbeddings} establishes that every $3$-connected cubic planar graph can be embedded on the projective plane, the torus and the Klein bottle, providing both lower and upper bounds on the number of inequivalent embeddings for each of these surfaces. However, a closer examination of the proofs reveals that the lower bounds are not satisfied in the case of strong embeddings. This aligns with the observations based on the data presented in Tables~\ref{table:projective}, \ref{table:torus} and \ref{table:Kleinbottle}.

In the following we show which graphs do not have strong embeddings on the projective plane, the torus or the Klein bottle.
For this, we need the following definition: 
A graph $G$ is \textbf{cyclically $k$-edge connected} if there is no set $A\subseteq E(G)$ of at most $k{-}1$ edges such that the graph $G\setminus A$ has at least two connected components having a cycle.

\begin{corollary}
    Let $G$ be a $3$-connected cubic planar graph with $|V(G)|\geq 5$.
    \begin{itemize}
        \item[1)] If $G$ is bipartite or cyclically $4$-edge connected, then $G$ has no strong embedding on the projective plane.
        \item[2)] If $G$ is cyclically $5$-edge connected, then $G$ has no strong embedding on the torus and the Klein bottle.
        \item[3)] If $G$ is not cyclically $5$-edge connected, but cyclically $4$-edge connected, then $G$ has a strong embedding on the torus.
    \end{itemize}
\end{corollary}
\begin{proof}
    \begin{itemize}
        \item[1)] Based on Theorem \ref{theorem:main} we know that $G^\ast$ must have a subgraph isomorphic to $K_4$ to ensure that $G$ can be strongly embedded on the projective plane. If $G$ is bipartite then $G^\ast$ is an even triangulation which does not have a subgraph isomorphic to $K_4$. If $G$ is cyclically $4$-edge connected then $G^\ast$ has no separating cycle of length three and so $G^\ast$ does not have a subgraph isomorphic to $K_4$.
        \item[2)] In Theorem~\ref{theorem:main} the subgraphs of $G^\ast$ that are required to obtain a strong embedding of $G$ on the torus and the Klein bottle are characterized.
        The graphs $K_{2,2,2},A_3,A_5$ and $A_6$ all contain a separating cycle of length four. Hence, if $G^\ast$ has a subgraph isomorphic to one of these graphs, then removing the dual edges of this cycle show that $G$ is not cyclically 5-edge connected. 
        In the proofs of Corollaries 19 and 20 from \cite{EnamiEmbeddings} it is shown that $G^\ast$ cannot have a subgraph isomorphic to $K_{2,m}$ for $m\geq 3$. Thus, the only remaining possibility is that $G^\ast$ contains a subgraph isomorphic to $K_{2,2}$. Since $G$ is cyclically 5-edge connected, no vertex of $G^\ast$ can lie in the interior of this $K_{2,2}$. This would require two vertices of the same partition set of the $K_{2,2}$ to be adjacent, which is incompatible with a strong embedding, as shown in Lemma~\ref{lemma:K2mstrong}. Hence, $G$ has no strong embedding on the torus and the Klein bottle.
        \item[3)] Since $G$ is not cyclically $5$-edge connected but cyclically 4-edge connected, there exist four edges whose removal disconnects the graph into exactly two components, each of which contains a cycle. Let denote the set of these four edges by $\mathcal{T}$. Thus, $H_\mathcal{T}$ has to be isomorphic to $K_{2,2}$, where the vertices in the same partition set are not adjacent in $G^\ast$. Then, $G$ has a strong embedding on the torus defined by $H_\mathcal{T}$. This proof proceeds analogously to the proofs of Corollaries 19 and 20 from~\cite{EnamiEmbeddings}.
    \end{itemize}
\end{proof}

Since the dodecahedral graph is the smallest graph that is cyclically 5-edge connected, it is consequently the smallest graph that fails to admit a strong embedding on the Klein bottle.
Since \cite{EnamiEmbeddings} shows that an upper bound on the number of subgraphs isomorphic to $K_4$ in a triangulation exists, we get an upper bound for strong re-embeddings on the projective plane.

\begin{corollary}
    A $3$-connected cubic planar graph $G$ has at most $\frac{\vert V(G)\vert}{2}-1$ inequivalent strong embeddings on the projective plane.
\end{corollary}
\begin{proof}
    In \cite[Lemma 15]{EnamiEmbeddings}, the author shows that every triangulation $T$ on the sphere has at most $\vert V(T)\vert -3$ subgraphs which are isomorphic to $K_4$. The dual graph $G^{\ast}$ is a triangulation and $\vert V(G^{\ast})\vert=\frac{1}{2}\vert V(G)\vert+2$. Thus, we obtain that $G^{\ast}$ has at most $\frac{\vert V(G)\vert}{2}-1$ subgraphs which are isomorphic to $K_4$. The statement follows with Theorem~\ref{theorem:main}.
\end{proof}

As noted by \cite{EnamiEmbeddings}, all embeddings of a given $3$-connected cubic planar graph on the projective plane can be computed in polynomial time and thus also all strong embeddings. This raises the question of whether a similar computational efficiency holds for strong embeddings on the torus and the Klein bottle. Such efficiency is achievable provided that the number of relevant subgraphs in the dual graph grows polynomially with the size of the graph.

Let $G$ be a $3$-connected cubic planar graph, as usual.
\cite{EnamiEmbeddings} shows that if the dual graph $G^{\ast}$ contains a subgraph isomorphic to $K_{1,1,m}$, then $G$ admits at least $2^m-1$ inequivalent embeddings on both the torus and the Klein bottle. However, the subgraphs isomorphic to $K_{1,1,m}$ for $m\geq 1$ and those isomorphic to $K_{2,m}$ for $m\geq 1$, as constructed in Enami's proof, do not correspond to strong embeddings. Consequently, that argument does not extend to the case of strong embeddings. Nevertheless, as the next theorem shows, there can still be exponentially many inequivalent strong embeddings on the torus and the Klein bottle.

\begin{theorem}
    There exists a $3$-connected cubic planar graph $G$ with exponentially many inequivalent strong embeddings on the torus and the Klein bottle in the size of $G$.
\end{theorem}
\begin{proof}
    Let $G$ be the $3$-connected cubic planar graph consisting of an inner and an outer cycle each of length $2m$ for $m\geq 2$, where the two cycles are connected as shown in Figure~\ref{fig:GExponential}. Then $G^{\ast}$ contains $K_{2,2m}$ as a subgraph, where the vertices of the partition set of size $2m$ are adjacent such that they form a cycle. The number of subgraphs isomorphic to $K_{2,m}$ in $K_{2,2m}$ is $\binom{2m}{m}$. With Stirling's approximation we obtain that $\binom{2m}{m}$ is asymptotically equivalent to $\frac{1}{\sqrt{\pi m}}4^m$, which grows exponentially in $m$. 
    Thus, $G^{\ast}$ has exponentially many subgraphs isomorphic to $K_{2,m}$. This directly implies an exponential number of strong embeddings on the torus and the Klein bottle of $G$, depending on whether $m$ is even or odd, by Theorem~\ref{theorem:main}. Therefore, the upper bounds for strong embeddings of $3$-connected cubic planar graphs on the torus and the Klein bottle is exponential.
\end{proof}
\begin{figure}[ht!]
    \centering
    \begin{tikzpicture}[scale=2]
        \tikzset{knoten/.style={circle,fill=black,inner sep=0.7mm}}
        \node [knoten] (a1) at (0,0.15) {};
        \node [knoten] (b1) at (0.8,0.75) {};
        \node [knoten] (c1) at (0.8,1.5) {};
        \node [knoten] (d1) at (-0.8,0.75) {};
        \node [knoten] (e1) at (-0.8,1.5) {};

        \node [knoten] (a2) at (0,0.6) {};
        \node [knoten] (b2) at (0.5,1) {};
        \node [knoten] (c2) at (0.5,1.5) {};
        \node [knoten] (d2) at (-0.5,1) {};
        \node [knoten] (e2) at (-0.5,1.5) {};
        
        \node (u) at (0,1.8) {$\dots$};
        
        \draw[-,thick] (a1) to (b1);
        \draw[-,thick] (b1) to (c1);
        \draw[-,thick] (a1) to (d1);
        \draw[-,thick] (d1) to (e1);

        \draw[-,thick] (a2) to (b2);
        \draw[-,thick] (b2) to (c2);
        \draw[-,thick] (a2) to (d2);
        \draw[-,thick] (d2) to (e2);

        \draw[-,thick] (a1) to (a2);
        \draw[-,thick] (b1) to (b2);
        \draw[-,thick] (c1) to (c2);
        \draw[-,thick] (d1) to (d2);
        \draw[-,thick] (e1) to (e2);

        \draw[-,thick] (e1) to (-0.5,1.75);
        \draw[-,thick] (e2) to (-0.23,1.7);

        \draw[-,thick] (c1) to (0.5,1.75);
        \draw[-,thick] (c2) to (0.23,1.7);
    \end{tikzpicture}
    \caption{Graph $G$ with two connected cycles of length $2m$}
    \label{fig:GExponential}
\end{figure}

\section{Conclusion}
In this paper, we have characterized strong graph embeddings of $3$-connected cubic planar graphs on the projective plane, the torus and the Klein bottle. These results enable us to make statements about the existence or non-existence of strong embeddings of a given $3$-connected cubic planar graph on one of these three surfaces.
Compared to Enami's results for general graph embeddings, we have shown that for strong embeddings we need to consider fewer types of subgraphs of the dual graph. This makes the calculation simpler. It may be possible to extend the approach to surfaces of higher genus, but it will be more complicated.
The results presented here serve as a foundation for further significant developments in the theory of strong embeddings. In particular, we demonstrate in a forthcoming paper—building on the current work—that Apollonian duals precisely characterize the class of $3$-connected cubic planar graphs admitting a unique strong embedding on orientable surfaces.

\acknowledgements
We gratefully acknowledge the funding by the Deutsche Forschungsgemeinschaft (DFG, German Research Foundation) in the framework of the Collaborative Research Centre CRC/TRR 280 “Design Strategies for Material-Minimized Carbon Reinforced Concrete Structures – Principles of a New Approach to Construction” (project ID 417002380).
The authors also thank Reymond Akpanya for his very valuable comments and advice. We also thank the anonymous reviewers for their suggestions, which were very helpful.

\bibliographystyle{abbrvnat}
\bibliography{lit.bib}
\label{sec:biblio}

\end{document}